\documentclass[english]{amsart}

\usepackage[all,cmtip]{xy}

\usepackage{amsmath,amssymb,amscd,amsfonts}

\usepackage{babel}
\usepackage{amstext}
\usepackage{amsmath}
\usepackage{amsfonts}
\usepackage{latexsym}
\usepackage{ifthen}
\usepackage{xypic}
\xyoption{all}
\newcommand{\cI}{{\mathcal I}}

\newcommand{\codim}{{\rm codim}}

\newcommand\ep{\varepsilon}

\newcommand{\Dl}{{\Delta}}

\newcommand{\lam}{{\lambda}}

\newcommand{\vph}{{\varphi}}
\newcommand{\sg}{{\sigma}}
\newcommand{\Th}{{\Theta}}

\newcommand\wtil{\widetilde}
\newcommand\what{\widehat}

\newcommand\ol{\overline}
\newcommand\ot{\otimes}
\newcommand\wed{\wedge}

\newcommand\cE{{\mathcal E}}

\newcommand\cF{{\mathcal F}}

\newcommand\sL{{\mathcal L}}

\newcommand\bR{{\mathbb R}}

\newcommand\bC{{\mathbb C}}
\newcommand\bQ{{\mathbb Q}}

\newcommand\bP{{\mathbb P}}
\newcommand\cO{{\mathcal O}}
\newcommand\CE{{\mathcal E}}
\newcommand\CO{{\mathcal O}} 

\newcommand\CS{{\mathcal S}}
\newcommand\CQ{{\mathcal Q}}
\newcommand\CF{{\mathcal F}}
\newcommand\CI{{\mathcal I}}

\def\Psh{\mathop{\rm Psh}\nolimits}

\def\Re{\mathop{\rm Re}\nolimits}
\def\Vol{\mathop{\rm Vol}\nolimits}
\def\rank{\mathop{\rm rank}\nolimits}

\def\End{\mathop{\rm End}\nolimits}
\def\Psh{\mathop{\rm Psh}\nolimits}

\def\Proj{\mathop{\rm Proj}\nolimits}

\def\sm{\mathop{\rm sm}\setminus}

\def\Sym{\mathop{\rm Sym}\nolimits}

\def\tr{\mathop{\rm tr}\nolimits}
\def\Supp{\mathop{\rm Supp}\nolimits}
\def\dbar{\overline\partial}
\def\ddbar{\partial\overline\partial}
\def\cO{{\mathcal O}}

\def\cE{{\mathcal E}}

\def\cF{{\mathcal F}}

\let\ol=\overline
\let\sm=\setminus

\let\Dl=\Delta
\let\ot=\otimes
\let\what=\widehat
\let\ep=\varepsilon
\let\wt=\widetilde
\let\wh=\widehat
\let\vph=\varphi
\let\Th=\Theta
\let\lam=\lambda
\def\bQ{{\mathbb Q}}
\def\bC{{\mathbb C}}
\def\bR{{\mathbb R}}

\def\bP{{\mathbb P}}

\newcounter{lemma}
\renewcommand{\thelemma}{\strut\kern-3pt\arabic{section}.\arabic{lemma}}
\newtheorem{lemma1}[lemma]{\setcounter{equation}{0}}
\let\saveref=\ref
\def\ref#1{\strut\kern3pt{\saveref{#1}}}
\def\eqref#1{({\saveref{#1}})}

\newenvironment{lemma}{\begin{lemma1}{\bf Lemma.}}{\end{lemma1}}
\newenvironment{example}{\begin{lemma1}{\bf Example.}\rm}{\end{lemma1}}
\newenvironment{theorem}{\begin{lemma1}{\bf Theorem.}}{\end{lemma1}}

\newenvironment{corollary}{\begin{lemma1}{\bf Corollary.}}{\end{lemma1}}
\newenvironment{remark}{\begin{lemma1}{\bf Remark.}\rm}{\end{lemma1}}
\newenvironment{definition}{\begin{lemma1}{\bf Definition.}}{\end{lemma1}}
\newenvironment{conjecture}{\begin{lemma1}{\bf Conjecture.}}{\end{lemma1}}

\newenvironment{ass}{\begin{lemma1}{\bf Assumption.}}{\end{lemma1}}

\theoremstyle{plain}
  \newtheorem{thm}{Theorem}[subsection]
  \newtheorem{thm'}{Theorem}[section]
  
  \newtheorem{prop'}[thm']{Proposition}
  \newtheorem{lem}[thm]{Lemma}
  
\theoremstyle{definition}

  \newtheorem{notations}[thm]{Notations}
  
  \newtheorem{rem}[thm]{Remark}
  \newtheorem*{acknowledgement}{Acknowledgement}

\begin{document}
\title[]{Singular Hermitian metrics and positivity of \\ 
direct images of pluricanonical bundles}

\author{Mihai P\u aun}

\address{Mihai P\u{a}un, Korea Institute for Advanced Study\\
85 Hoegiro, Dongdaemun-gu\\
Seoul 130-722, South Korea,}
\email{paun@kias.re.kr}
\date{\today}
\thanks{}

\begin{abstract} This is an expository article.
In the first part we recall the definition and 
a few results concerning singular Hermitian metrics on torsion-free coherent 
sheaves. They offer the perfect platform for the study of properties of 
direct images of twisted pluricanonical bundles which we will survey in the second part.     
\end{abstract}
\maketitle

\tableofcontents

\section{Introduction}

\smallskip

\noindent In birational classification programs 
\cite{KM}, \cite{Cam04} elaborated and successfully implemented in algebraic geometry,
the positivity properties of the canonical bundle 
$$\displaystyle K_X:= \wedge^{\dim(X)}\Omega^1_X$$ of a projective manifold $X$
play a central role. In this context, \emph{positivity} refers e.g.\ to the size of Kodaira dimension $\kappa(X)$, or the existence of holomorphic sections of multiples of approximate $\bQ$-bundles $K_X+ \ep A$, where $A$ is ample 
and $\ep> 0$ is a rational number.

\noindent Next, suppose that we are given a family of manifolds 
\begin{equation}
p:X\to Y
\end{equation}
--instead of an isolated one-- such that the 
total space $X$ and the base $Y$ are non-singular. Then 
the analogue of $K_X$ for $p$ is the so-called
relative canonical bundle 
\begin{equation}
K_{X/Y}:= K_X- p^\star(K_Y).
\end{equation} 
As we see from this definition,
the restriction $\displaystyle K_{X/Y}|_{X_y}$ to a non-singular fiber $X_y= p^{-1}(y)$ of $p$ identifies naturally with the canonical bundle $\displaystyle K_{X_y}$ 
of the fiber. Moreover, many results/conjectures are indicating that the singularities of $p$ and the variation of the
complex structure of its fibers are encoded in the geometry of the 
relative canonical bundle. In conclusion, 
understanding the positivity properties of $K_{X/Y}$ is a 
question of critical importance. 
\smallskip

\noindent ``In practice'', one studies the twisted version of this bundle, i.e.
$\displaystyle K_{X/Y}+ L$, where $(L, h_L)$ is a Hermitian $\bQ$-line bundle 
endowed with a positively curved (possibly singular) metric $h_L$. Usually, the twist
$L$ corresponds to an effective snc $\bQ$-divisor $\sum (1-\nu^j)W_j$ on $X$, 
where $\nu^j\in [0, 1)\cap \bQ$, but this is not always the case.
\smallskip

\noindent A technique which turned out to be very successful in this context 
consists in considering the direct image sheaf
\begin{equation}\label{mess2}
\cE_m:= p_\star\left(m(K_{X/Y}+ L)\right)
\end{equation} 
where $m$ is positive integer, such that $mL$ is a line bundle.
There is an impressive body of articles dedicated to the study of $\cE_m$ via methods 
arising from Hodge theory: \cite{Gr}, \cite{Ft}, \cite{Ka82},
\cite{Ka02}, \cite{Ka09}, \cite{Ko86}, \cite{Ko07}, \cite{Vi1}, as well as
\cite{MT08}, \cite{MT09}, \cite{Fn} and \cite{PS} in recent years. The reader can profitably consult the excellent survey \cite{Ho}, where some of these results are discussed.  
\smallskip

\noindent In the article \cite{B}, the $L^2$ theory combined with methods of
complex differential geometry is successfully used in the
study the direct image sheaf \eqref{mess2}: this can be seen as the starting point of 
the quantitative analysis of its positivity properties. 

\noindent In this text we will present some of the metric properties of the direct image of $m(K_{X/Y}+ L)$ by following \cite{BPDuke}, \cite{Raufi1} and 
\cite{mpst}. The main result we obtain in 
\cite{mpst} states that \emph{the direct image sheaf $\cE_m$ admits a 
positively curved singular Hermitian metric.} Our objective in what follows is to 
explain the main notions/results leading to the proof of this statement.
This survey is organized
as follows.
\smallskip

\noindent
In section two we recall the definition and a few results concerning 
singular Hermitian vector bundles. The guideline is provided by an important theorem 
due to J.-P. Demailly \cite{Dnote} establishing the equivalence 
between the pseudo-effectivity  of a line bundle and the existence of a closed positive current in its 
first Chern class (as recalled in 2.1). The notion of positively curved
singular Hermitian vector 
bundle (and more generally, of a torsion-free coherent sheaf), 
as it has emerged from the articles \cite{BPDuke}, \cite{Raufi1}, 
\cite{Raufi2} and \cite{mpst}, is recalled in sections 2.3 and 2.4 respectively.
Next, the vector bundle counterpart of the pseudo-effectiveness is
the notion of weak positivity, introduced by E.~Viehweg, cf.\ \cite{Vbook}. 
In section 2.5 we show that a sheaf $\cF$ admitting a positively curved 
singular Hermitian metric is weakly positive--this 
is a result from \cite{mpst} and it represents a --one sided only, \emph{h\'elas}...-- 
higher rank analogue of the result of J.-P.~Demailly mentioned above.
\smallskip

\noindent
In section three we analyze the metric properties of 
the direct image sheaves $p_\star(K_{X/Y}+ L)$. In this 
direction, the foundation theorem was
obtained by Bo~Berndt\-sson in \cite{B}. The main result in \cite{B} is 
stated (and proved) in case of a submersion $p$ and for a Hermitian line bundle 
$(L, h_L)$ whose metric is smooth and positively curved: among many other things,
B.\ Berndtsson shows that the direct image $p_\star(K_{X/Y}+ L)$ admits a metric whose curvature form is semi-positive in the sense of Griffiths. 
In this section (cf.\ 3.2 and 3.3) we provide a complete overview of our results in \cite{mpst} as already mentioned: this is a wide generalization of \cite{B} in the context of algebraic fiber spaces $p$, and for a singular Hermitian twisting $(L, h_L)$. 
\smallskip

\noindent By combining the main results of sections two and three, we infer that E.\ Viehweg's 
weak positivity of direct images is induced by the metric properties of these sheaves.
As in the case of a line bundle, the slight advantage of the analytic 
point of view is that it provides a positively curved 
singular Hermitian metric \emph{on the vector bundle itself}, and not only on its approximations 
by algebraic objects (this turned out to be important 
for applications, cf.\ \cite{jcmp}). By contrast, it may happen very well that a weakly positive vector bundle does not admit global holomorphic sections (neither any
of its symmetric powers).   
\smallskip

\noindent We conclude in section four with the brief discussion 
of a very recent result due to S.~Takayama, cf.\ \cite{taka2}.  
\smallskip

\noindent Finally, we would like to mention the very interesting article \cite{HPS} which offers an alternative point of view to many of the results discussed here.

\begin{acknowledgement}
I am grateful to the organizers of the
2015 AMS Summer Institute in Algebraic Geometry for the kind
invitation to give a lecture and to contribute 
the present survey to the proceedings. 
The preparation of this work was 
supported by the National Science Foundation
under Grant No. DMS-1440140 while the author was in residence at the
Mathematical Sciences Research Institute in Berkeley, California, during the
Spring 2016 semester.
\end{acknowledgement}

\section{Singular Hermitian metrics on vector bundles}\label{SsHm}

\noindent In this first part of our survey we present a 
few results concerning the notion of \emph{singular Hermitian metric} on  
vector bundles of arbitrary rank, together with their generalization in the context of
torsion free sheaves.
Our main sources are \cite{mpst}, \cite{BPDuke}, \cite{Raufi1}, \cite{Raufi2}.  
As a preparation for this, we first 
recall some basic facts concerning singular metrics on line bundles
and Griffiths positivity of vector bundles.  
\smallskip

\subsection{Line bundles} Let $X$ be a non-singular projective
manifold. We denote by
\begin{equation}\label{eq10}
\cE_X\subset H^{1,1}(X, \bR)
\end{equation} 
the closure of the convex cone generated by the first Chern classes $c_1(D)$ of
$\bQ$-effective divisors of $X$. A line bundle    
$L\to X$ is \emph{pseudo-effective} (``psef'' for short in what follows) 
in the sense of algebraic geometry if 
\begin{equation}\label{eq11}
c_1(L)\in \cE_X.
\end{equation}
It follows from the definition that 
a line bundle $L$ is psef provided that given any ample line bundle $A$ and any 
positive integer $m$, some multiple of the bundle 
$mL+ A$ admits global holomorphic sections. 
As shown by J.-P. Demailly, this property has a perfect 
metric counterpart, which we will next discuss, cf.\ \cite{Dnote}, \cite{Dpark}
and the references therein. 

Let $\Omega\subset \bC^n$ be the unit ball. A function $\vph:\Omega\to [-\infty, \infty)$ is \emph{plurisubharmonic} (``psh'' for short) provided that the 
following requirements are fulfilled.
\begin{enumerate}
\item[(i)] The function $\vph$ is upper-semicontinuous.
\smallskip

\item[(ii)] The mean value inequality holds true locally at each 
point $x_0\in \Omega$, i.e.
\begin{equation}\label{eq12}
\vph(x_0)\leq \int_0^{2\pi}\vph\big(x_0+ e^{\sqrt{-1}\theta}\xi\big)\frac{d\theta}{2\pi}
\end{equation} 
where $\xi\in \bC^n$ is arbitrary such that $|\xi|\ll 1$ in order to insure that
$x_0+ e^{\sqrt{-1}\theta}\xi\in \Omega$ for any $\theta\in [0, 2\pi]$.
\end{enumerate}

\noindent We denote by $\Psh(\Omega)$ the set of psh functions defined on $\Omega$. If
$\vph\in \Psh(\Omega)$ is smooth, then one can show that 
\begin{equation}\label{eq13} 
\sqrt{-1}\ddbar \vph\geq 0
\end{equation}
at each point $x_0$ of $\Omega$ (this is a consequence of the property (ii) above, cf.\ \cite{Dagbook}, chapter one), in the sense that we have 
\begin{equation}\label{eq14} 
\sum\frac{\partial^2\vph}{\partial z_j\partial \ol z_k}(x_0)v_j\ol v_k\geq 0
\end{equation}
for any vector $v\in \bC^n$.

An important fact about the class of psh functions is that \eqref{eq13} holds true 
\emph{without} the smoothness assumption provided that it is interpreted in the sense
of currents. This can be seen as follows: let 
\begin{equation}\label{eq15} 
\vph_\ep:= \vph\star \rho_\ep
\end{equation} 
be the regularization of $\varphi$ by convolution with a radial kernel $\rho_\ep$. By (ii)
we infer that $\vph_\ep$ is also psh for each $\ep$, and the family of smooth psh functions $\displaystyle (\vph_\ep)_{\ep> 0}$ is monotone (although defined on a slightly smaller domain). 
As $\ep\to 0$, it turns out that for each pair of indexes $(j, k)$ 
the mixed second order derivative
\begin{equation}\label{eq16}
\frac{\partial^2\vph_\ep}{\partial z_j\partial \ol z_k}
\end{equation}
is converging towards a complex measure denoted by $\vph_{j\ol k}$, and the pointwise 
inequality \eqref{eq14} becomes
\begin{equation}\label{eq17} 
\sum\vph_{j\ol k}v_j\ol v_k\geq 0
\end{equation}
for any choice of functions $v_j$ (as a measure on $\Omega$). One can see that the 
measures $\vph_{j\ol k}$ are independent of the particular regularization we 
choose, and we have 
\begin{equation}\label{eq18}
\sqrt{-1}\ddbar\vph= \sqrt{-1}\sum \vph_{j\ol k}dz_j\wedge d\ol z_k
\end{equation}
in the sense of distributions.
As a conclusion, the inequality \eqref{eq13} holds true in the context of arbitrary psh functions.
\medskip

\noindent Returning to the global context, a \emph{singular metric} $h$ on a line bundle 
$L$ is given locally on coordinate sets $\Omega_\alpha$ by 
\begin{equation}\label{eq19} 
h|_{\Omega_\alpha}= e^{-\varphi_\alpha}|\cdot |^2
\end{equation} 
where $\vph_\alpha\in L^1_{\rm loc}(\Omega_\alpha)$ are called local weights of $h$. 
We remark that one can define the \emph{curvature} of $(L, h)$ by the same formula 
as in the usual case
\begin{equation}\label{eq20}
\Theta_h(L)|_{\Omega_\alpha}= \frac{\sqrt -1}{\pi}\ddbar\vph_\alpha;
\end{equation}
this is a consequence of the hypothesis $\vph_\alpha\in L^1_{\rm loc}(\Omega_\alpha)$. 
The important difference is that in this degree of generality the curvature 
$\displaystyle \Theta_h(L)$ becomes a closed (1,1)-current 
(rather than a differential form). 
\medskip

\noindent We recall that a line bundle $L$ is 
\emph{psef in metric sense} if there exists a metric $h$ on $L$ 
such that the corresponding curvature current is positive. In this context we have the
following important result, due to J.-P.~Demailly in \cite{Dnote}.

\begin{theorem}\cite{Dnote}\label{dem11}
Let $L$ be a line bundle on a projective manifold. Then $c_1(L)\in \cE_X$ if and 
only if $L$ is psef in metric sense, i.e. there exists $T$ a closed positive current of 
(1,1)--type
such that $T\in c_1(L)$.
\end{theorem}    
\begin{remark} The ``if'' part of Theorem \ref{dem11} can be seen as generalization of the celebrated Kodaira embedding theorem: the existence of a singular metric on $L$ with positive curvature current
implies the fact that for any ample line bundle $A$ we have 
$H^0\big(X, k(mL+ A)\big)\neq 0$ provided that $m\geq 0$ and for any $k\gg 0$. 
The proof relies on $L^2$ theory for the $\dbar$ operator. 

The other implication is more elementary, as follows. Let $F$ be a line bundle, and let 
$\displaystyle (\sigma_j)_{j=1,\dots N}$ be a family of holomorphic sections of some 
multiple $mF$ of $F$. Locally on each $\Omega_\alpha$ the section $\sigma_j$ corresponds 
to a holomorphic function $f_{\alpha, j}$ and then we define a metric $h$ on $F$ via its 
local weights
\begin{equation}\label{eq21}
\vph_\alpha:= \frac{1}{m}\log\sum |f_{\alpha, j}|^2.
\end{equation}
It may well happen that the metric $h$ defined in \eqref{eq21} 
is singular, but in any case the corresponding curvature current is positive, cf.\ 
\cite{Dagbook}. 

In a similar manner, the sections of multiples of $F_m:= mL+ A$ are inducing a metric on $L$, by the formula
\begin{equation}\label{eq22}
\vph_{\alpha, m}:= \frac{1}{km}\Big(\log\sum |f_{\alpha, j}^{(km)}|^2- k\varphi_{\alpha A}\Big).
\end{equation} 
where $f_{\alpha, j}^{(km)}$ correspond to a family of sections of $k(mL+ A)$ and 
$\displaystyle \varphi_{\alpha A}$ is a smooth metric on $A$. The curvature of the metric
$h_m$ defined by \eqref{eq22} is greater than $-\frac{1}{m}\Theta(A)$. By an appropriate 
normalization, we can assume that the weak limit of $\vph_{\alpha, m}$ as $m\to \infty$ 
is not identically $-\infty$, and it defines therefore a metric on $L$ whose curvature 
current is semi-positive.   
\end{remark}
\smallskip

\noindent In the next sections we will discuss the higher rank vector bundles analogue
of the results presented here.


\subsection{Vector bundles}
\medskip

\noindent A vector bundle $E$ on a projective manifold $X$ is \emph{weakly positive} if there exist a positive integer $k_0$, together with an ample line bundle $A$ such that the evaluation map
\begin{equation}\label{eq23}
H^0\big(X, S^{m}(S^{k}E\otimes A)\big)\to S^{m}(S^{k}E\otimes A)
\end{equation}
is generically surjective, for any $k\geq k_0$ and $m\geq m_1(k)$, where $m_1(k)$ is a positive integer depending on $k$. 
This notion can be formulated in the more general context of torsion 
free sheaves; an important example is provided by the following deep result due to E. Viehweg. 
\vskip2mm
\begin{theorem}\label{Vi1} \cite{Vi1} Let $f: X\to Y$ be a proper surjective morphism of projective manifolds with connected fibers. If $m$ is a positive integer, then the direct image sheaf $f_\star(mK_{X/Y})$ is weakly positive. 
Here we denote by $mK_{X/Y}$ the $m^{\rm th}$ tensor power of the relative canonical bundle.
\end{theorem}
\vskip2mm

\begin{remark}\label{dps}
If the rank of $E$ is one, then we see that $E$ is weakly positive if and only if $c_1(E)\in \cE_X$. 
Therefore, it is reasonable to think that the weak positivity if the higher rank analogue of the 
pseudo-effectiveness in the sense of algebraic geometry. In fact, 
according to \cite{DPS}, the weak positivity of $E$ is equivalent to the fact that the tautological bundle $\cO_E(1)$ on $\bP(E^\star)$ is pseudo-effective, and the union of all curves 
$C\subset \bP(E^\star)$ such that $\displaystyle \cO_E(1)\cdot C< 0$ is contained is an analytic set which is vertical with respect to the projection $\bP(E^\star)\to X$. Here $\bP(E^\star)$ is the space of lines in the dual bundle $E^\star$ of $E$.   
\end{remark}
\smallskip

\noindent From the perspective of Theorem \ref{dem11}, the weak positivity of a vector 
bundle should be equivalent to some positivity property of the curvature tensor associated to
a ``singular metric'' on $E$. In order to make this a bit more precise, 
we recall next the definition of 
Griffiths positivity of a Hermitian vector bundle $(E, h)$.

Let $X$ be a complex manifold, and let $E\to X$ be a vector 
bundle of rank $r\geq 1$, endowed with a smooth Hermitian metric $h$. We denote by 
$$\Theta_h(E)\in {\mathcal C}^\infty_{1,1}\big(X, {\rm End}(E)\big)$$
the curvature form of $(E, h)$.
We refer to \cite[Ch.\ 3, Ch.\ 10]{Dnote} for the definition of this tensor and its basic properties.
\smallskip

\noindent Let $x_0\in X$ be a point, and let $(z_1,..., z_n)$ be a system of local coordinates on $X$ centered at 
$x_0$. We consider a local holomorphic frame $e_1,..., e_r$ of $E$ near $x_0$, orthonormal at $x_0$.  
The Chern curvature tensor can be locally expressed as follows
$$\Theta_h(E)= \sqrt{-1}\sum c_{j\ol k \lambda\ol \mu}dz_j\wedge d\ol z_k\otimes e_\lambda^\star\otimes e_\mu$$
where $j, k= 1,..., n$ and $\lambda, \mu= 1,..., r$. We say that $(E, h)$ \emph{is semi-positive in the sense of Griffiths}
at $x_0$ if 
\begin{equation}\label{eq24}
	\sum c_{j\ol k \lambda\ol \mu}\xi_j\ol \xi_kv_\lambda\ol v_\mu\geq 0
\end{equation}
for every $\displaystyle \xi\in T_{X, x_0}$ and $v\in E_{x_0}$. 
The Hermitian bundle $(E, h)$ is semi-positive in the sense of Griffiths if it satisfies the property above at each point of $X$. The semi-negativity of a Hermitian vector bundle is defined by reversing the sense of the inequality in \eqref{eq24}.

\begin{remark}\label{sing_1} It is well-known (cf.\ \cite[10.1]{Dnote}) that $(E, h)$ is semi-positive if and only if its dual $(E^\star, h^\star)$ is semi-negative in the sense of Griffiths. 
This important duality property is unfortunately not verified by the positivity in the sense of Nakano, which however will not be discussed here.
\end{remark}

\begin{remark}\label{obs} Let $\bP(E^\star)$ be the projectivization of the dual of $E$, by which we mean the space of 
lines in $E^\star$ and let $\displaystyle \cO_E(1)$ be the corresponding tautological line bundle. 
The metric $h$ induces a metric on $\displaystyle \cO_E(1)$ and (almost by definition) we see that if the bundle $(E, h)$ is semi-negative, then
$$ 	\log \vert v\vert_h^2 $$
is plurisubharmonic (psh, for short), for any local holomorphic section $v$ of the bundle 
$\cO_E(1)$.
This observation will be crucial from our point of view, since it gives the possibility of \emph{defining} the notion of Griffiths positivity for a bundle $(E, h)$ without referring to the curvature tensor!
\end{remark}

\smallskip


\noindent It turns out that if $(E, h)$ is semi-positive in the sense of Griffiths, then $E$ is weakly positive (we will prove this statement later in a more general context). As for the converse direction,
the situation is far from clear even if $E$ is ample instead 
of weakly positive. We recall in this respect the following important conjecture proposed by
Ph.\ Griffiths \cite{Gr}, a long standing and deemed difficult problem.

\begin{conjecture}\cite{Gr}
Let $E$ be an ample vector bundle, in the sense that $\cO_E(1)$ is ample on $\bP(E^\star)$. Then $E$ admits 
a smooth Hermitian metric $h$ such that $(E, h)$ is positive in the sense of Griffiths.   
\end{conjecture}

\noindent The notion of \emph{singular Hermitian metric} and the 
corresponding Griffiths positivity of a vector bundle
we will discuss next were formulated so that 
the following two results hold true.
\begin{enumerate}

\item[(a)] A singular Hermitian vector bundle $(E, h)$ which is positively curved 
in the sense of 
Griffiths is weakly positive.
\smallskip

\item[(b)] Let $p: Y\to X$ be an algebraic fiber space such that the Kodaira
dimension of its generic fibers is not $-\infty$.
Then  
the direct image sheaf $\cE:= p_\star(mK_{X/Y})$ admits a positively curved 
singular Hermitian metric 
(in the sense of Griffiths). 

\end{enumerate}

\noindent In particular, we infer that E.\ Viehweg's weak positivity of direct image
$\cE$ (cf.\ (b) and Theorem \ref{Vi1} above) is ``generated'' by the existence of a positively curved metric
on this sheaf.

\subsection{Singular Hermitian metrics: basic properties}
In this subsection we recall the formal definition of a singular Hermitian metric on a vector bundle, and we discuss some of its main properties.

Let $E\to X$ be a holomorphic vector bundle of rank $r$ on a complex manifold $X$.
We denote by
\begin{equation}\label{eq2}
H_{r}:=\{A=(a_{i\ol j})\}
\end{equation} the set of $r \times r$, semi-positive definite Hermitian matrices. Let $\overline H_r$ be the space of 
semi-positive, possibly unbounded Hermitian forms on $\bC^r$.
The manifold $X$ is endowed with the Lebesgue measure.
\medskip

\noindent We recall the following notion.

\begin{definition}\label{sHm}
A {\it singular Hermitian metric} $h$ on $E$ is given locally by a measurable map with values in 
$\overline H_{r}$ such that 
\begin{equation}\label{eq1}
0<\det h<+\infty
\end{equation} 
almost everywhere.
\end{definition}

\noindent In the definition above, a matrix valued function $h=(h_{i\ol j})$ is measurable provided that all entries $h_{i\ol j}$ are measurable.
We note that in the paper \cite[p.\,357]{BPDuke} the condition $0<\det h<+\infty$ almost everywhere is not part of the definition of a {singular Hermitian metric}
on a vector bundle. However, as it was highlighted in \cite{Raufi1}, \cite{Raufi2}, 
this additional condition concerning the determinant is important because of its connection with the notion of
\emph{curvature tensor} associated to a singular Hermitian metric. 
\smallskip

\noindent Let $(E, h)$ be a vector bundle endowed with a 
singular Hermitian metric $h$. Given a {local section} $v$ of $E$, i.e. an element $v \in H^{0}(U,E)$ defined on some open subset $U\subset X$, the function 
$\displaystyle |v|_{h}^2 : U \to \bR_{\ge 0}$ is measurable, given by 
\begin{equation}
|v|_{h}^{2} = {}^{t}v h \ol v=\sum h_{i\ol j}v^{i}\ol{v^{j}}
\end{equation} 
where $v={}^{t}(v^{1},\ldots,v^{r})$ is a column vector.
\medskip

\noindent Following \cite[p.\,357]{BPDuke}, we recall next the notion of positivity/negativity of a singular Hermitian vector bundle as follows.

\begin{definition}\label{curv}
Let $h$ be a singular Hermitian metric on $E$.
\smallskip

\begin{enumerate}

\item[(1)] The metric $h$ is {\it semi-negatively curved} if the function
\begin{equation}
x\to \log |v|_{h, x}^{2}
\end{equation}
is psh for any local section $v$ of $E$.
\smallskip

\item[(2)] The metric $h$ is {\it semi-positively curved} if the dual singular Hermitian metric $h^{\star}:={}^{t}h^{-1}$ on the dual vector bundle $E^{\star}$ is semi-negatively curved.

\end{enumerate}

\end{definition}

\noindent Few remarks are in order at this point.

\begin{remark}\label{bounded}

\smallskip

\noindent $\bullet$ 
If $h$ is a smooth Hermitian metric on $E$ 
in the usual sense the requirement \ref{curv}(1) (resp.\ (2)) is nothing but the classical Griffiths semi-negativity (resp.\ Griffiths semi-positivity) of $h$ as we have recalled in the previous subsection.
\smallskip

\noindent $\bullet$ A-priori a singular Hermitian metric $h$ is only defined almost everywhere, because its coefficients are measurable functions. However, if $h$ is semi-positively curved (or semi-negatively curved), then the coefficients $h_{i\ol j}$ are unambiguously defined at each point of $X$. In the semi-negative case, this can be seen as follows.

Let $v= {}^{t}(1,0,\ldots,0)$ is a local section of $E$ with constant coefficients with respect to a fixed holomorphic frame, then $|v_{1}|_{h}^{2}=h_{1\ol 1}$ is semi-positive and psh. In particular, the function $h_{1\ol 1}$ is well-defined at each point, and it is also locally bounded from above. Next, if we take 
$v={}^{t}(1,1,0,\ldots,0)$, then we have 
\begin{equation}\label{eq25}
|v|_{h}^{2}=h_{1\ol 1}+h_{2\ol 2}+2\text{Re}\, h_{1\ol 2}.
\end{equation}
For $v'={}^{t}(1,\sqrt{-1},0,\ldots,0)$, we have 
\begin{equation}\label{eq26}
|v'|_{h}^{2}=h_{1\ol 1}+h_{2\ol 2}+2\text{Im}\, h_{1\ol 2}.
\end{equation} 
By combining \eqref{eq25} and \eqref{eq26}, we infer that $h_{1\ol 2}$ is everywhere defined, and
\begin{equation}\label{eq27}
|h_{1\ol 2}|^{2}\le h_{1\ol 1}h_{2\ol 2} \le \frac{1}{2}\big(h_{1\ol 1}^2+h_{2\ol 2}^2\big)
\end{equation}
\smallskip

\noindent $\bullet$ If $(E, h)$ is semi-negatively curved, then there exists a constant $C> 0$ such that we have $|h_{i\ol j}|\leq C$, in the sense that we can find a covering of $X$ with co-ordinate open sets on which the inequality above holds. This is a consequence of the previous bullet. In the dual case, i.e. if
$(E, h)$ is semi-positively defined, then the coefficients of $\displaystyle {}^th^{-1}$ are bounded from above.
\end{remark}

\smallskip

\begin{remark} One could define an alternative notion of ``positively curved singular Hermitian metric'' on a vector bundle $E$ by requiring that the tautological line bundle $\cO_E(1)$ on
$\bP(E^\star)$ is pseudo-effective, and such that the Lelong level sets of a closed positive 
curvature current in $\displaystyle c_1\left(\cO_E(1)\right)$
do not project onto $X$ (inspired by \cite{DPS}, cf.\ Remark \ref{dps}). This would correspond to a singular Finsler metric with positive curvature on    
$E$. However, it is unclear whether this induces a positively curved singular Hermitian
metric on $E$; we refer to \cite{D92}, \cite{B} for some results in this direction.  
\end{remark}
 
\medskip

\noindent The following regularization statement (cf.\ \cite{BPDuke}) is an important technical tool, despite of its simplicity.

\begin{lemma}\label{appr} \cite[3.1]{BPDuke} Let  $X\subset \bC^n$ be a polydisc, 
and suppose $h$  is a semi-negatively curved (resp.\ semi-positively curved) singular Hermitian metric on $E$.
Then, on any smaller polydisc there exists a sequence of smooth Hermitian metrics $\{h_{\nu}\}_{\nu}$
decreasing (resp.\ increasing) pointwise to $h$ whose corresponding curvature tensor is Griffiths negative (resp.\ positive).
\end{lemma}

\noindent In the statement above the sequence $\{h_{\nu}\}_{\nu}$ is called \emph{decreasing} in the sense that
the sequence of functions $\{|s|_{h_{\nu}}^{2}\}_{\nu}$ is decreasing for any constant section $s$, or equivalently, the matrix corresponding to each difference $h_{\nu}-h_{\nu+1}$ is semi-positive definite. 
We sketch next the proof of the lemma.
\begin{proof}
Assume that $h$ is semi-negatively curved, and that the bundle $E$ is trivial. 
We first remark that we can add a small multiple of the flat metric on $E$, so that $h$ is strictly positive;
by this operation, the semi-negativity property is conserved. Then we define
\begin{equation}\label{eq28}
h_\ep(z)= \int\rho_\ep(w- z)h(w)
\end{equation}
where $\rho_\ep$ is the standard convolution kernel. Then we have
\begin{equation}\label{eq29}
|\xi|^2_{h_\ep}= \int\rho_\ep(w- z)|\xi|^2_h
\end{equation}
for any local section $\xi$ with constant coefficients, and $h_\ep$ is a smooth Hermitian metric. It is equally semi-negative, as it follows directly from the expression \eqref{eq29}. 
The monotonicity assertion is a consequence of Jensen convexity inequality.
\end{proof}
\medskip

\noindent As illustration of the usefulness of Lemma \ref{appr}, we will state next a few properties 
of singular Hermitian vector bundles with semi-positive/semi-negative curvature. Needless to say, these results are part of the standard differential geometry of vector bundles  in the non-singular case.

\begin{lemma}\label{tensors}\cite{Raufi1}, \cite{mpst}
Let $(E, h)$ be a vector bundle endowed with a positively curved 
singular Hermitian metric. Then the induced metric on $\Sym^mE$ and $\Lambda^q E$ respectively are positively curved as well. In particular, the determinant $\det E$ admits a metric $\det h$ for which the 
corresponding curvature current is positive.   
\end{lemma}
  
\begin{proof} We only sketch here the proof of the fact that $\det E$ is psef. 
Let $\{h_{\nu}\}_{\nu}$ be the sequence of metrics approximating $h$, with the properties stated in
\ref{appr}. We define $\vph_{\nu}:= -\log\det h_{\nu}$, the weight of the induced metric on the determinant bundle; it is psh (and smooth), and $\vph_{\nu}$ is decreasing to $\vph_{h}:= -\log\det h$.
Thus $\vph_{h}$ is psh and in particular $\vph_{h} \in L^{1}_{\rm loc}(X,\bR)$. 

The other assertions of Lemma \ref{tensors} can be checked in a similar way, so we provide no 
further details.
\end{proof}

\noindent Next, we see that the important \emph{curvature decreasing} property of a 
subbundle is preserved in singular context.

\begin{lemma}\label{subquot}\cite{Raufi1}, \cite{mpst}
Let $h$ be a singular Hermitian metric on $E$. The following assertions are true.
\smallskip
\begin{enumerate}
\item[\rm (1)] Let $S \subset E$ be a subbundle.
Then the restriction $h_S:=h|_S$ defines a singular Hermitian metric on $S$, and $h_S$ is 
semi-negatively curved if $h$ is.
\smallskip

\item[\rm (2)]
 Let $E\to Q$ be a quotient vector bundle.
Suppose that $h$ is semi-positively curved.
Then $Q$ has a naturally induced semi-positively curved singular Hermitian metric $h_Q$.
\end{enumerate}
\end{lemma}

\begin{proof} We will briefly discuss the point (1), since (2) is a consequence of (1) together with the duality properties of the positivity in the sense of Griffiths. 

In the first place, the condition 
\begin{equation}\label{30}
0<\det h<\infty
\end{equation}
together with the fact that $h(x)\in H_r$ for a.e. $x\in X$ shows that 
\begin{equation}\label{31}
0<\det h_S<\infty
\end{equation}
a.e. on $X$. Next, let $v$ be a local section of $S$. Via the inclusion $S\subset E$ given in (1) we can see $v$ as local section of $E$, and then the function
\begin{equation}\label{32}
\log |v|_h^2
\end{equation} 
is psh; since we have $\displaystyle |v|_{h_S}= |v|_h$, the conclusion follows.
\end{proof}

\noindent  Our next statement concerns the 
behavior of the Griffiths positivity/negativity in singular context with respect to inverse images.

\begin{lemma}\label{pullback}\cite{mpst}
Let $f:Y\to X$ be a proper holomorphic surjective map between two complex manifolds, and let $E$ be a vector bundle on $X$.
\smallskip
\begin{enumerate}

\item[\rm (1)] We assume that $E$ admits a singular Hermitian metric $h$ with semi-negative (resp.\ semi-positive) curvature. Then $f^{\star}h$ is a singular Hermitian metric on $f^\star (E)$, and it is semi-negatively (resp.\ semi-positively) curved.
\smallskip

\item[\rm (2)]
Let $X_{1}$ be a non-empty Zariski open subset, and let $Y_{1}=f^{-1}(X_{1})$ be its inverse image.
We consider $h_{1}$ a singular Hermitian metric on $\displaystyle E_{1}=E|_{X_{1}}$.
We assume that the singular Hermitian metric $f^{\star}h_{1}$ on $f^{\star}E_{1}$ extends as a singular Hermitian metric $h_{Y}$ on the inverse image bundle $f^{\star}E$ with 
semi-negative (resp.\ semi-positive) curvature. Then $h_{1}$ extends as a singular Hermitian metric on $E$ with semi-negative (resp.\ semi-positive) curvature.
\end{enumerate}
\end{lemma}

\noindent We remark that the statement above holds even if we do not assume that the map $f$ is surjective, provided that the pull-back $f^{\star}h$ is well-defined. Also, if we have $X_{1}=X$ in (2) 
then we see that the metric $h_{1}$ is semi-negatively (resp.\ semi-positively) curved if and only if $f^{\star}h_{1}$ is.

\begin{proof}
We will discuss the negatively curved case.
\smallskip

\noindent (1) 
Let $U$ be a coordinate open subset in $X$ such that $E|_{U} \cong U \times \bC^{r}$, and $U' \subset$ a smaller open subset so that an approximation result \ref{appr} holds for $h$.
Let $\{h_{\nu}\}$ be a decreasing sequence towards $h$ as in \ref{appr}.
We consider an open subset $V \subset f^{-1}(U')$ and a non-zero section $v \in H^{0}(V,f^{\star}E)$.
Then the sequence $\{\log |v|_{f^{\star}h_{\nu}}^{2}\}_{\nu}$ is decreasing to $\log |v|_{f^{\star}h}^{2} \not\equiv -\infty$.
Since $f^{\star}h_{\nu}$ is smooth and semi-negatively curved, $\log |v|_{f^{\star}h_{\nu}}^{2}$ is psh.
Thus so is its limit $\log |v|_{f^{\star}h}^{2}$. 
\smallskip

\noindent (2)
Let $u\in H^{0}(U,E)$ be any local section.
Then $f^{\star}u\in H^{0}(f^{-1}(U), f^{\star}E)$.
Then by assumption, $\log |f^{\star}u|_{h_{Y}}^{2}$ is psh on $f^{-1}(U)$.
In particular $\log |f^{\star}u|_{h_{Y}}^{2}$ is bounded from above on any relatively compact subset of $f^{-1}(U)$.
For any $x \in U\cap X_{1}$ and a point $y \in f^{-1}(x)$, we have 
$|f^{\star}u|_{h_{Y}}^{2}(y) = |u|_{h_{1}}^{2}(x)$.
Thus we see that the psh function $\log |u|_{h_{1}}^{2}$ a-priori defined only on $U\cap X_{1}$ is bounded from above on $U' \cap X_1$, where $U'\subset U$ is any relatively compact subset. 
Hence, it can be extend as a psh function on $U$, 
so the metric $h_1$ indeed admits an extension $h$ on $X$ in such a way that the Hermitian bundle $(E, h)$ is semi-negatively curved.
\end{proof}
\smallskip

\begin{remark}
In the context of Lemma \ref{pullback}, it would be interesting to know if the following more general statement holds true. \emph{Assume that there exists a singular Hermitian metric $h_Y$ on $f^\star E$ which is 
positively curved. Can one construct a singular Hermitian metric $h_X$ on $E$ such that $(E, h_X)$ is positively curved?} It is likely that the answer is ``yes''.
\end{remark}
\medskip

\subsection{Singular Hermitian metrics on torsion free sheaves}

It turns out that very interesting objects, such as direct images of adjoint bundles 
\begin{equation}
\cE_m:= f_{\star}(mK_{X/Y}+L)
\end{equation} 
do not always have a vector bundle structure everywhere; nevertheless,
it is important to have a notion of ``singular Hermitian metric" on $\cE_m$.
More generally, we introduce in this subsection a notion of (metric) positivity 
for torsion free sheaves. It turns out that
the theory is essentially the same as the vector bundle case. 

\smallskip

\noindent Let $\cE$ be a coherent, torsion free sheaf on a complex manifold $X$.  
We denote by $\displaystyle X_{\cE}\subset X$ the maximum Zariski open subset of $X$ such that
the restriction of $\cE$ to $X_{\cE}$ is locally free.
Since $\cE$ is torsion free, we have $\codim (X\setminus X_{\cE}) \ge 2$. Hence the following definition is meaningful.

\begin{definition}\label{sHm3} Let $\cE$ be a coherent, torsion free sheaf on a complex manifold $X$. 
\begin{enumerate}
\item[(1)]
A {\it singular Hermitian metric} $h$ on $\cE$ is a singular Hermitian metric on the vector bundle $\displaystyle \cE|_{X_{\CE}}$.
\smallskip

\item[(2)]
A singular Hermitian metric $h$ on $\cE$ is {\it semi-positively curved}, 
(resp.\ {\it semi-negatively curved})
if the restriction $\displaystyle \left(\cE|_{X_{\cE}}, h|_{X_{\cE}}\right)$
is a semi-positively curved (resp.\ semi-negatively curved) 
singular Hermitian vector bundle.
\smallskip

\item[(3)]
$\cE$ is {\it semi-negatively} (resp.\ {\it semi-positively}) {\it curved}, if it admits a 
semi-negatively (resp.\ semi-positively) curved singular Hermitian metric.
\end{enumerate}
\end{definition}

\begin{remark}\label{dual} We have the following comments about the definition \ref{sHm3}
\smallskip

\noindent $\bullet$ Let $\cE$ be a coherent, torsion free sheaf on $X$, and let $h$ be a singular Hermitian metric on $\cE$.
Then $h$ induces a metric on $\cE^{\star\star}$, the bi-dual of $\cE$, given that 
\begin{equation}
\cE|_{X_{\cE}} \cong \cE^{\star\star}|_{X_{\cE}}.
\end{equation} 
Therefore $\cE$ is semi-negatively/semi-positively curved if and only if $\cE^{\star\star}$ is 
 semi-negatively/semi-positively curved. 
\smallskip

\noindent $\bullet$ Let $U$ be any open subset of $X$. 
If $h$ is a semi-negatively curved singular Hermitian metric on $\cE$ then for any 
$v \in H^{0}(U,\CE)$ the function
\begin{equation}\label{eq30}
\log |v|_{h}^{2} 
\end{equation}
is psh on $U\cap X_{\CE}$.
Since we have $\codim (X\sm X_{\CE}) \ge 2$ the function \eqref{eq30} above 
extends as a psh function on $U$. 
\smallskip

\end{remark}
\medskip

\noindent We have next the ``sheaf version'' of Lemma \ref{subquot}.

\begin{lemma}\label{subquot2}
Let $h$ be a singular Hermitian metric on the coherent, torsion free sheaf $\CE$.

\begin{enumerate}
\smallskip

\item[\rm(1)] Let $\CS \subset \CE$ be a subsheaf.
Then the restriction $h_\CS:=h|_\CS$ defines a singular Hermitian metric on $\CS$. 
If $h$ is semi-negatively curved, then so is $h_\CS$.
\smallskip

\item[\rm(2)]  Let $\CE\to \CQ$ be a quotient torsion free sheaf.
Suppose that $h$ is semi-positively curved.
Then $\CQ$ has a naturally induced singular Hermitian metric $h_\CQ$ with semi-positive curvature.
\smallskip

\item[\rm(3)]  Let $\CF$ be a torsion free sheaf, and suppose there exists a sheaf homomorphism $a : \CE\to\CF$ which is generically surjective.
Suppose that $h$ is semi-positively curved.
Then $\CF$ has a naturally induced singular Hermitian metric $h_{\CF}$ with semi-positive curvature.
\end{enumerate}
\end{lemma}

\begin{proof} We will only discuss the points (1) and (3).

\noindent (1) 
By restricting everything on the maximum Zariski open subset where $\CS$ is locally free, we may assume $\CS$ is locally free.
Let $W$ be a Zariski open subset such that $\CE$ is locally free and $\CS$ is a subbundle of $\CE$.
Then $h_\CS=h|_\CS$ defines a singular Hermitian metric on $\CS$ over $W$, and hence induces a singular Hermitian metric on $\CS$ over $X$ by \ref{dual}.
Let $U$ be an open subset and $v \in H^{0}(U,\CS) \subset H^{0}(U,\CE)$.
Then on $U\cap W$, $|v|_{h_{\CS}}^{2}=|v|_{h}^{2}$ (in the right hand side, $|v|_{h}^{2}$ is measured as local section of the sheaf $\CE$, i.e.\ $v\in H^{0}(U,\CE)$).
If $h$ is negatively curved, then $\log |v|_{h}^{2}$ is a.e.\,psh on $U$, and hence $h_{\CS}$ is a.e.\,negatively curved.
\smallskip 

\noindent (3)
By dualyzing the map $a$, we obtain a sheaf injection $\CF^{\star}\to \CE^{\star}$.
Thanks to the point (1), we obtain a singular Hermitian metric $h^{\star}|_{\CF^{\star}}$ on $\CF^{\star}$ 
with negative curvature. This metric is only defined on the set $W_{\cE}\cap W_{\CS}$, but it 
extends to $W_{\CS}$ by the procedure explained in Remark \ref{dual}.
By taking the dual, we have a singular Hermitian metric $(h^{\star}|_{\CF^{\star}})^{\star}$ (not exactly same, but the discrepancies eventually occur on a measure zero set only) on $\CF^{\star\star}$ with positive curvature.
We then use \ref{dual} again to obtain a singular Hermitian metric on $\CF$.
\end{proof}

\subsection{Griffiths semi-positivity and weak positivity}
\smallskip

\noindent Let $\cF$ be a coherent, torsion free sheaf on a variety $X$; 
we denote by $S^{m}(\CF)$ the $m$-th symmetric tensor product of $\CF$ so 
that $S^{0}(\CF) = \CO_{Y}$, and let $\what{S}^{m}(\CF)$ be the double dual of the sheaf $S^{m}(\CF)$.

\begin{definition}\label{dd-ample}
Let $X$ be a smooth projective variety, and let $\CF$ be a torsion free coherent sheaf on $X$.
Let $A$ be a fixed ample divisor on $X$.
Then $\CF$ is {\it weakly positive} if there exists a Zariski open subset $X_0\subset X$ such that 
for any integer $a > 0$, there exists an integer $b>0$ such 
that 
\begin{equation}\what S^{ab}(\CF)\ot\CO_{X}(bA)
\end{equation} 
is generated by global sections at each point $x\in X_0$. 
\end{definition}
\smallskip

\noindent The following statement represents a partial generalization of Theorem \ref{dem11}
for arbitrary rank vector bundles.

\begin{theorem}\label{Gp imply wp2}\cite{mpst}
Let $X$ be a smooth projective variety, and let $\CE$ be a coherent, torsion free sheaf on $X$.
Suppose that $\CE$ admits a positively curved singular Hermitian metric $h$.
Then $\CE$ is weakly positive.
\end{theorem}

\begin{proof}

 \noindent As a preliminary discussion, we recall some facts due to \cite[V.3.23]{Nbook}.
 Let $\pi : \bP(\CE) \to X$ be the scheme associated to $\CE$, together with its tautological 
 line bundle $\CO_{\CE}(1)$. Formally we have $\displaystyle 
\bP(\CE) = \Proj \big(\bigoplus_{m \ge 0} S^m(\CE)\big)$.
Let $X_{\cE} \subset X$ be the maximum Zariski open subset of $X$ such that the restriction 
$\displaystyle E:= \CE|_{X_\cE}$ is locally free. Then we have 
\begin{equation}\label{eq38}
\pi^{-1}(X_\cE)\cong \bP(E^\star)
\end{equation}
and the restriction of the tautological bundle to the set \eqref{eq38} identifies 
with $\cO_{E}(1)$.

Let $\mu:Y \to \bP(\CE)$ be the desingularization of the component of $\bP(\CE)$ containing $\pi^{-1}(X_\cE)$; we can assume that $\mu$ is an isomorphism over $\pi^{-1}(X_\cE)$.
We denote by $f=\pi\circ\mu : Y \to X$ the resulting morphism, and
we define
\begin{equation}\label{eq32}
	Z:= f^{-1}(X\setminus X_\cE),\quad Y_1:= f^{-1}(X_\cE);
\end{equation}
then we can assume that $Z \subset Y$ is a divisor.
We denote by $L=\mu^{\star}\CO_{\CE}(1)$ the inverse image of the tautological bundle on
$\bP(\cE)$. Let $A$ be an ample line bundle on $X$. Then there exists an effective $\bQ$-divisor
$W$ on $Y$, such that the following two properties are verified:
\begin{enumerate} 
\smallskip

\item[($a_1$)] We have $\Supp (W)\subset Z$, so that via the map $f$ the divisor $W$ projects in codimension at least two.
\smallskip

\item[($a_2$)] There exists a positive rational number $\ep_0$ such that the $\bQ$-divisor
\begin{equation}\label{eq35}
f^\star A+ \ep_0 L- W
\end{equation}
is ample. 
\end{enumerate}
These properties are standard, and we are referring to \cite[V.3.23]{Nbook} for 
a full justification --e.g. one 
can use the fact that $\cE$ is a quotient of a vector bundle $V$ 
(since $X$ is projective) and then the scheme $\bP(\cE)$ is embedded in the manifold
$\bP(V^\star)$. 
\smallskip

In particular, there exists an integer $m_0\gg 0$ such that 
\begin{equation}\label{36}
\sL_0:= m_0\big(f^\star A+ \ep_0 L- W\big)- K_Y
\end{equation} 
is an ample line bundle. In what follows, we will show that the restriction morphism
\begin{equation}\label{37}
H^0\big(Y_1, K_Y+ \sL_0+ mL|_{Y_1}\big)\to H^0\big(Y_x, K_Y+ \sL_0+ mL|_{Y_x}\big) 
\end{equation}
is surjective, for any $m\geq 0$. Here the point $x\in X_\cE$ is assumed to 
belong to the set $(\det h< \infty)$ and  
we denote by $Y_x:= f^{-1}(x)$ the fiber of $f$ over $x\in X$. The surjectivity of
\eqref{37} will be a consequence of the $L^2$ theory, as follows.
\smallskip

\noindent We recall that the sheaf $\cE$ is endowed with a positively curved metric $h$. 
This induces a metric $h_L$ on the bundle $\displaystyle L|_{Y_1}$, such that
\begin{enumerate}
\smallskip

\item[(i)] \emph{The curvature current $\displaystyle \Theta_{h_L}(L|_{Y_1})\geq 0$ 
is semi-positive.} This is a direct consequence of the fact that $h$ is positively 
curved.
\smallskip

\item[(ii)] \emph{The multiplier ideal $\cI(h_L^{\otimes m}|_{Y_x})$ 
of the restriction of $h_L^{\otimes m}$ to the fiber $Y_x$ is trivial for any $m\geq 0$,
provided that $x\in (\det h< \infty)$.} This is a consequence of the third bullet
of Remark \ref{bounded}: given that $h$ is positively curved, 
its eigenvalues at $x$ are bounded from below away from zero. If in addition the determinant of $h(x)$ is finite, then the eigenvalues of $h(x)$ are bounded from above as well.
In particular the corresponding multiplier sheaf is trivial.  
\end{enumerate}
\smallskip

\noindent Our next observation is that the manifold $Y_1$ is complete K\"ahler. Indeed, by \eqref{eq38} the manifold $Y_1$ can be identified with $\bP(E^\star)$, where $E^\star$ is a vector bundle over $X_1$. Now the manifold $X_1$ carries a complete K\"ahler metric, say $\omega_1$ 
cf.\ \cite{D82}, and a small multiple of the curvature of $\cO_E(1)$ plus the inverse image of $\omega_1$ will be a complete K\"ahler metric on $Y_1$. 

Thus we are in position to apply Ohsawa-Takegoshi theorem, cf.\ \cite{OT}: the 
bundle $\sL_0+ mL$ is endowed with a metric whose local weight is $\vph_0+ m\vph_L$, where
$e^{-\vph_0}$ is a smooth 
metric with positive curvature on $\sL_0$ (the existence of such object is insured by \eqref{36} above). The curvature conditions in \cite{OT} are clearly satisfied, as soon as
$m_0$ is large enough --which we can assume--. And by the point (ii) above the 
integrability requirements are automatically satisfied. This shows that the morphism 
\eqref{37} is indeed surjective. 

\noindent In conclusion, we have shown that for any $m\gg 0$ 
the global holomorphic sections of the bundle
\begin{equation}
S^mE\otimes A^{m_0}|_{X_\cE}
\end{equation}
are generating the stalk $S^mE\otimes A^{m_0}_x$. These objects 
extend automatically as sections of $\what S^m\cE\otimes A^{m_0}$. We have therefore 
showed that the global generation property holds true for any point of the set
\begin{equation}
X_\cE\cap \big(\det h<\infty\big).
\end{equation}
By taking large enough tensor powers of $\what S^m\cE\otimes A^{m_0}$ we see that the same global generation property holds true on a Zariski open set,   
and the proof of Theorem \ref{Gp imply wp2} is finished.
\end{proof}

\medskip

\noindent We state next a more complete form of Theorem \ref{Gp imply wp2}, and we refer to
\cite{mpst} for the proof.

\begin{theorem} \label{bigness}
Let $X$ be a smooth projective variety, and let $\CF$ be a torsion free coherent sheaf on $X$.
Let $\bP(\CF)$ be the scheme over $X$ associated 
to $\CF$, say $\pi : \bP(\CF) \to X$, and let $\CO_{\CF}(1)$ be the tautological line bundle on $\bP(\CF)$.
Suppose that $\CO_{\CF}(1)|_{\pi^{-1}(X_{\cF})}$ admits a singular Hermitian metric $g$ with semi-positive curvature. We assume 
that there exists a point $y \in X_\cF$ such that $\CI(g^k|_{\bP(\CF_y)})=\CO_{\bP(\CF_y)}$ for any $k>0$, where $\bP(\CF_{y})=\pi^{-1}(y)$.
Then
\smallskip 

\noindent 
{\rm (1)} $\CF$ is weakly positive at $y$.
\smallskip 

\noindent 
{\rm (2)} Assume moreover that there exists an open neighborhood $W$ of $y$ and a K\"ahler form $\eta$ on $W$ such that $\displaystyle 
\Th_{g}\left(\cO_{\cF}(1)\right) - \pi^{\star}\eta \ge 0$ on $\pi^{-1}(W)$, then $\CF$ is big.
\end{theorem}
\begin{remark} 
It does not seem to be know wether the weak positivity of $E$ implies the Griffiths semi-positivity 
in the sense of Definition \ref{sHm}
(this would be a singular version of Griffiths conjecture). 
\end{remark}
\medskip

\subsection{Curvature} An important ingredient in the theory of singular Hermitian 
line bundles is the \emph{curvature current}: as briefly recalled in subsection 2.1,
this is a $(1, 1)$-form with measure coefficients. Given a singular Hermitian 
vector bundle $(E, h)$, one would expect the curvature 
$\Theta_h(E)$ to be a (1,1)--form with measure
coefficients and values in $\End(E)$ at least when $E$ is, say, semi-negatively curved.  
Unfortunately, this is not the case, as the next example due to H. Raufi shows it. 
\begin{example} 
Following \cite[Theorem 1.3]{Raufi1}, let 
$E=\bC \times \bC^2$ be the trivial bundle of rank two on $\bC$. 
We consider the metric on $E$ given by the following expression
$$
h=\begin{pmatrix}
1+|z|^{2}& z \\
\ol z& |z|^{2}
\end{pmatrix}
$$
Then the coefficients of the Chern connection $\partial_E:= h^{-1}\partial h$ \emph{are not} in $L^1_{\rm loc}$ near the origin, and thus the curvature cannot be defined as a vector valued (1,1)-current with measure coefficients.
\end{example}
\smallskip

\noindent The following result is a particular case of \cite{Raufi1}; it gives a sufficient criteria 
in order to define 
the notion of curvature current associated to $(E, h)$ which is very useful.

\begin{theorem} {\rm (\cite[Thm 1.6]{Raufi1})}\label{singsing} 
Let $(E, h)$ be a positively curved singular Hermitian vector bundle of rank $r$. 
We denote by $h_\ep$ the sequence of metrics on $E$ obtained by approximation as in 
\eqref{eq28}. The following assertions hold true.
\begin{enumerate}
\smallskip

\item[(i)] We assume that $(E, h)$ is semi-negatively curved, and that $\det h\geq \ep_0> 0$
for some positive real number $\ep_0$. Then the $L^2$ norm of the connection form 
$\displaystyle h_\ep^{-1}\partial h_\ep$ is uniformly bounded (with respect to $\ep$), and
we have
\begin{equation}\label{eq40}
h_\ep^{-1}\partial h_\ep\to h^{-1}\partial h
\end{equation}
in weak sense as $\ep\to 0$. In particular, the coefficients of the connection form 
$\displaystyle h^{-1}\partial h$ belong to $L^2$, so we can define the curvature  
\begin{equation}\label{eqq41}
\Theta_h(E):= \dbar \big(h^{-1}\partial h\big)
\end{equation}  
in the sense of currents. Moreover, the coefficients of the curvature form $\displaystyle \Theta_{h_\ep}(E)$
have a uniform $L^1_{\rm loc}$ bound, and we have 
\begin{equation}
\Theta_{h_\ep}(E)\to \Theta_h(E)
\end{equation} 
in weak sense, as $\ep\to 0$. Therefore, the current $\Theta_h(E)$ is of order zero, 
i.e. it has measure coefficients, and it is positive in the sense of Griffiths. 
\smallskip

\item[(ii)] We assume that $(E, h)$ is positively curved, and that $\det h\leq 
\ep_0^{-1}< \infty$. Then the same conclusion as in {\rm (i)} holds.
\end{enumerate}
 
\end{theorem}
\medskip

\noindent We provide here a few explanations about the statement \ref{singsing}. The fact that 
$\displaystyle \Theta_{h}(E)$ is a matrix-valued $(1,1)$-current of order zero 
means that locally on 
some coordinate set $U$ centered at some point
$x\in X$ we have
\begin{equation}\label{singcurv}
\Theta_{h}(E)|_U= \sqrt{-1}\sum_{j, k, \alpha,\beta}\mu_{j\ol k\alpha\ol \beta}dz^j\wedge dz^{\ol k}
\otimes e_\alpha^\star\otimes e_\beta
\end{equation}
where $\displaystyle \mu_{j\ol k\alpha\ol \beta}$ are \emph{measures} on $U$ (rather than smooth functions as in the 
classical case), $(e_\alpha)_{\alpha=1,\dots , r}$ is a local holomorphic frame of $E$ and $(z^i)_{i=1,\dots ,n}$ are local coordinates. 
The positivity in the sense of Griffiths we are referring to in Theorem \ref{singsing}
means that for any local holomorphic vector field $\displaystyle \sum v^j\frac{\partial}{\partial z^j}$ and 
for any local holomorphic section $\displaystyle \sum \xi^\alpha e_\alpha$, the measure
\begin{equation}\label{measure}
\sum \mu_{j\ol k\alpha\ol \beta}v^j\ol{v^k}\xi^\alpha\ol {\xi^\beta}
\end{equation}
is (real and) positive on $U$. 

\noindent Also, we remark that the equality \eqref{eqq41} \emph{does not} imply automatically that $\Theta_h(E)$ has measure coefficients--it is the last part of 
Theorem \ref{singsing} who shows that this is the case.

\begin{proof}
We will only sketch here the main steps of the proof, and we refer to 
\cite{Raufi1} for a complete argument. The first remark is that it is enough to 
prove the statement (i), because the second part (ii) is a direct consequence of (i) applied to the 
dual bundle $(E^\star, h^\star:= {}^th^{-1})$.

Let $v$ be a local section of $E$, defined on a coordinate open set 
$\Omega$. We have the equality
\begin{equation}\label{eq161401}
\sqrt{-1}\ddbar |v|_{h_\ep}^2= -\langle\Theta_{h_\ep}(E)v, v \rangle_{h_\ep}+
\sqrt{-1}\langle D^\prime v, D^\prime v \rangle_{h_\ep} 
\end{equation} 
where $D^\prime$ in \eqref{eq161401} is the (1,0) part of the Chern connection associated to $(E, h_\ep)$. Since by hypothesis the Hermitian bundle $(E, h)$ is semi-negatively defined, the family of psh functions $\displaystyle \big(|v|_{h_\ep}^2\big)_{\ep> 0}$ is bounded in $L^1_{\rm loc}$. Actually, we know more than that: for any relatively compact 
subset $\Omega^\prime\subset \Omega$ there exists a constant $C> 0$ such that 
\begin{equation}
\sup_{\Omega^\prime}|v|_{h_\ep}^2\leq C
\end{equation}
for any positive $\ep$. By hypothesis, the form $\displaystyle 
\langle\Theta_{h_\ep}(E)v, v \rangle_{h_\ep}$ is semi-positive, so we infer that we have
\begin{equation}\label{eq161402}
\int_{\Omega^\prime}\sqrt{-1}\langle D^\prime v, D^\prime v \rangle_{h_\ep}\wedge \rho\leq C
\sup_\Omega|\rho| 
\end{equation}
for any test form $\rho$, where the constant $C$ above is independent of $\ep$. Given that 
we have
$$D^\prime= \partial - h_\ep^{-1}\partial h_\ep\wedge $$ locally, the relation 
\eqref{eq161402} together with the fact that $|h_{\ep i\ol j}|\leq C$ show that we have 
\begin{equation}\label{eq161403}
\Vert \partial h_\ep\Vert^2\leq C
\end{equation} 
for some constant $C> 0$, where $\Vert \cdot\Vert$ above is the operator 
(Hilbert-Schmidt) norm of the corresponding endomorphism. 

So far we have not used the hypothesis $\det h\geq \ep_0$; it comes into the picture in the proof of the fact that we have 
\begin{equation}\label{eq161404}
 h_\ep^{-1}\partial h_\ep\to  h^{-1}\partial h
\end{equation}
as $\ep\to 0$ in the sense of distributions. This is done by evaluating 
separately the quantities $\displaystyle 
h_\ep^{-1}\partial h_\ep-  h^{-1}_\ep\partial h$ and 
$\displaystyle 
h_\ep^{-1}\partial h-  h^{-1}\partial h$.
The justification of \eqref{eq161404} will not be reproduced here, we will rather refer to cf.\ \cite{Raufi1}, 373- 375.
\smallskip

\noindent We already know that the coefficients of the trace 
\begin{equation}
\displaystyle \tr \Theta_{h_\ep}(E)
\end{equation}
are uniformly bounded in $L^1_{\rm loc}$ (with respect to $\ep$); by 
Griffiths negativity hypothesis, the same is true for the coefficients of 
$\displaystyle \Theta_{h_\ep}(E)$, as we shall see next.
We consider 
\begin{equation}\label{161405}
\Theta_{h_\ep}(E)|_U= 
\sqrt{-1}\sum_{j, k, \alpha,\beta}-\mu^{(\ep)}_{j\ol k\alpha\ol \beta}
dz^j\wedge dz^{\ol k}\otimes e_\alpha^\star\otimes e_\beta
\end{equation}
the local expression of the curvature of $(E, h_\ep)$. We know that
\begin{equation}\label{161406}
\big\Vert \sum_{\alpha, \beta} \mu^{(\ep)}_{j\ol k\alpha\ol \beta}h_\ep^{\alpha\ol\beta}\big\Vert_{L^1(\Omega^\prime)}\leq C
\end{equation}
for any $(j, k)$, uniformly with respect to $\ep$. In the expression \eqref{161406} we denote by 
$\displaystyle h_\ep^{\alpha\ol\beta}$ the coefficients of the inverse matrix corresponding to $h_\ep$. We note that 
the eigenvalues of $h_\ep$ are uniformly bounded from below.
Hence if we write \eqref{161406}
for $j= k$, then we obtain
\begin{equation}\label{161407}
\Vert\mu^{(\ep)}_{k\ol k\alpha\ol \alpha}\Vert_{L^1(\Omega^\prime)}\leq C
\end{equation}
for any $k, \alpha$, as a consequence of Griffiths negativity hypothesis. The same kind of arguments imply first that $\displaystyle \Vert\mu^{(\ep)}_{j\ol k\alpha\ol \alpha}\Vert_{L^1(\Omega^\prime)}\leq C$ for any $(j, k)$ and any $\alpha$, and finally that 
\begin{equation}\label{161408}
\Vert\mu^{(\ep)}_{j\ol k\alpha\ol \beta}\Vert_{L^1(\Omega^\prime)}\leq C
\end{equation}
for any indexes $j, k, \alpha$ and $\beta$. The convergence of 
$\displaystyle \Theta_{h_{\ep}}(E)$ towards $\Theta_h(E)$ shows that the current
 $\Theta_h(E)$ has \emph{measure} coefficients, which is what we wanted to prove. 
\end{proof}

\medskip

\noindent We have the following 
consequence of the previous results.

\begin{corollary}\label{det} \cite{jcmp} Let $(\cE, h)$ be a positively curved singular Hermitian 
sheaf, such that whose restriction to $W_{\cE}$ is a vector bundle of rank $r$, so that $\displaystyle \codim X\setminus W_{\cE}\geq 2$. The determinant line bundle 
$$\displaystyle \det \cE$$ admits a singular hermitian metric whose 
curvature current $\Theta$ is positive. Moreover, we have the following statements.
\begin{enumerate}

\item[(a)] If $\Theta$ is non-singular when restricted to some open subset 
$\Omega\subset W_{\cE}$, then the curvature
current of $\displaystyle \cE|_\Omega$ is well-defined.
\smallskip

\item[(b)] If $\Theta$ vanishes when restricted to the open subset
$\Omega^\prime\subset W_{\cE}$, then so does the full curvature tensor corresponding to $\displaystyle \cE$. In this case the metric $h|_{\Omega^\prime}$ is smooth.
\end{enumerate}

\end{corollary}

\begin{proof} 
The metric $h$ induces a metric $\det h$ on the determinant bundle 
$\displaystyle \det \cE|_{W_{\cE}}$.
whose curvature is semi-positive. 
It is well-known that psh functions extend across sets of codimension at least two, hence the the first part of the 
corollary follows. 
\smallskip

The statement (a) is a direct consequence of Theorem \ref{singsing}, 
because the metric induced on the 
determinant bundle on $\Omega$ is \emph{smooth}, by standard regularity results. 

\noindent As for part (b), we use 
Theorem \ref{singsing} again, and it implies that the restriction of the curvature current corresponding to
$\displaystyle \cE|_{\Omega^\prime}$ is well-defined.
We establish its vanishing next; as we will see, it is a consequence of 
the positivity in the sense of Griffiths of the curvature of 
$\displaystyle \cE$, 
combined with the fact that its trace $\Theta$ is equal to zero.
We remark at this point that is really important to have at our disposal the {curvature current} 
as given by Theorem \ref{singsing}, and not only the positivity in the sense of Griffiths.

A by-product of the proof of Theorem \ref{singsing} (cf. \cite[Remark 4.1]{Raufi1}) 
is the fact that the curvature current $\Theta$ of $\det E$ is simply the trace of the matrix-valued 
current $\displaystyle \Theta_{h}(E)$. By using the notations \eqref{singcurv}, this is equivalent to the fact that
\begin{equation}\label{mp1}
 \sum_{j, k}\sum_\alpha\mu_{j\ol k\alpha\ol \alpha}dz^j\wedge dz^{\ol k}= 0.
\end{equation}
Since $\displaystyle \Theta_{h}(\cE)$ is assumed to be positive in 
the sense of Griffiths, we infer that the current  
\begin{equation}\label{mp2}
 \sum_{j, k, \alpha}\mu_{j\ol k\alpha\ol \alpha}dz^j\wedge dz^{\ol k}
\end{equation}
is positive \emph{for each index $\alpha$}. When combined with \eqref{mp1}, this implies that
\begin{equation}\label{mp3}
\mu_{j\ol k\alpha\ol \alpha}\equiv 0
\end{equation} 
for each $j, k, \alpha$. But then we are done, since the positivity of 
$\displaystyle \Theta_{h}(\cE)$ together with \eqref{mp3} shows that for each pair of indexes $\alpha, \beta$
we have \begin{equation}\label{mp4}
\Re\Big(\xi^\alpha\xi^{\ol \beta}\sum_{j, k}\mu_{j\ol k\alpha\ol \beta}v^j v^{\ol k}\Big)\geq 0
\end{equation} 
(notations as in \eqref{measure}) which in turn implies that $\displaystyle \mu_{j\ol k\alpha\ol \beta}\equiv 0$
for any $j, k, \alpha, \beta$. 
The current $\displaystyle \Theta_{h}(\cE)|_{\Omega^\prime}$ is therefore identically zero.
\smallskip

\noindent The regularity statement is verified as follows. In the first place we already know that the coefficients of $h$
are bounded. This follows thanks to relation \eqref{??} which implies that the absolute value of the coefficients of the dual metric $h^\star$ is bounded from above, combined with the fact that 
the determinant $\det h$ is smooth. 

Since $\dbar$ of the connection form (= curvature current) is equal to zero,
it follows that the connection is smooth. Locally near a point of $\Omega^\prime$ we therefore have
\begin{equation}\label{eq4011}
\partial h= h\cdot \Psi
\end{equation}
where $\Psi$ is smooth. The relation \eqref{eq4011}
holds in the sense of distributions; by applying the $\dbar$ operator to it, we see 
that $h$ satisfies an elliptic equation. In conclusion, it is smooth. \end{proof}

\section{Metric properties of direct images}

\medskip

\noindent Let $p: X\to Y$ be an algebraic fiber space: by definition this
means that $X$ and $Y$ are non-singular and that $p$ is a projective map 
with connected fibers. The \emph{relative canonical bundle} 
$\displaystyle K_{X/Y}$ corresponding to the map $p$ is 
\begin{equation}
K_{X/Y}= K_X- p^\star K_Y
\end{equation}
where $K_X:= \wedge^{\dim X}\Omega_X$. 
When restricted to a generic fiber $X_y$
of $p$, the bundle $K_{X/Y}$ is identifies naturally with the canonical bundle 
$\displaystyle K_{X_y}$. Moreover, many results/conjectures 
are supporting the idea that the variation of the complex structures of 
the fibers of $p$ is reflected into the positivity properties of the relative canonical bundle. Therefore, it is of fundamental importance 
to study the algebraic and metric properties of this bundle, respectively.
In practice on always has to deal with 
the twisted version
\begin{equation}\label{msri1}
K_{X/Y}+ L
\end{equation}
of the bundle above,
where $L\to X$ is a line bundle endowed with a metric $h_L$ whose curvature 
current is semi-positive (actually, $L$ will be a $\bQ$-line bundle, but 
let us ignore that for the moment).

\noindent One way of studying the properties of the bundle \eqref{msri1} 
is via the direct image sheaf
\begin{equation}\label{msri2}
\cE:= p_\star(K_{X/Y}+ L)
\end{equation} 
on $Y$. In the next subsections we will analyze the metric properties of \eqref{msri2} in increasing degrees of generality. 

\subsection{A natural metric on $p_\star(K_{X/Y}+ L)$} To start with, we assume that 
$p:X\to Y$ is a smooth proper 
fibration 
i.e. a submersion between a K\"ahler manifold $X$ of
of dimension $m+n$, and a complex $m$-dimensional manifold $Y$ (which can be simply the unit 
ball in the Euclidean space $\bC^m$ for our immediate purposes). Let $(L, h_L)\to X$ be a holomorphic 
line bundle endowed with a \emph{non-singular} Hermitian metric $h_L$ whose curvature 
form is semi-positive. 

Given a point $y\in Y$, any section $\displaystyle u\in H^0\big(X_y, K_{X_y}+ L|_{X_y}\big)$ extends, in the sense that there exists an open coordinate set $y\in \Omega\subset Y$ and a section
\begin{equation}\label{msri3}
U\in H^0\big(p^{-1}\Omega, K_{X/Y}+ L|_{p^{-1}\Omega}\big)
\end{equation}
such that $\displaystyle U|_{X_y}= u\wedge dt$ (here we abusively denote by $dt$ the 
inverse image of 
a local generator $dt_1\wedge\dots \wedge dt_m$ of $K_Y$). This is a consequence of the OT theorem, cf.\ \cite{}, and it is at this point that the K\"ahler assumption for $X$ is used. 
In any case, we infer that the complex manifold
\begin{equation}
E:= {\cup}_{y\in Y}H^0\big(X_y, K_{X_y}+ L|_{X_y}\big)
\end{equation}
(which equals the total space of the direct image of $K_{X/Y}+ L$)
has a structure of vector bundle of rank $r:= h^0\big(X_y, K_{X_y}+ L|_{X_y}\big)$ over the base $Y$.
So the space of local -smooth- sections of $E|_{\Omega}$ are simply the sections of the bundle
$\displaystyle  K_{X/Y}+ L|_{p^{-1}\Omega}$ whose restriction to each fiber of $p$ is holomorphic. 
\smallskip

\noindent The vector bundle $E= p_\star(K_{X/Y}+ L)$ admits a natural \emph{complex structure}
which we now recall. Let $u$ be a section of $E$; then $u$ is holomorphic if 
\begin{equation}\dbar u\wedge dt= 0.
\end{equation}
 This is equivalent to saying that the section $u\wedge dt$ of the bundle
$\displaystyle  K_{X}+ L|_{p^{-1}\Omega}$ is holomorphic (in the usual sense).
\smallskip

\noindent The holomorphic bundle $E$ can be endowed with a \emph{Hermitian metric}, as follows. Let $u, v$ be two 
sections of $E$. We denote by $(u_y)$ the family of $L$-twisted holomorphic $(n, 0)$ forms on fibers of $p$ induced by $u$. Then the scalar product of $u$ and $v$ is given by
\begin{equation}\label{msri4}
\langle u, v\rangle_y:= \int_{X_y}c_nu_y\wedge \ol{v_y}e^{-\vph_L},
\end{equation}
where $c_n:= (-1)^{n^2/2}$ is the usual unimodular constant.
We denote by $h_E$ the $L^2$ metric on $E$ defined by \eqref{msri4} (in the article 
\cite{mpst} this is called Narasimhan-Simha metric). 
\medskip

\noindent The curvature of the Hermitian vector bundle $(E, h_{E})$ was 
computed by Berndtsson in \cite{B}; his impressive result states as follows.

\begin{theorem}\cite{B}
The Chern curvature form $\displaystyle \Theta_{h_{E}}(E)$ is positive 
in the sense of Griffiths. 
\end{theorem}
\noindent Actually, the theorem in \cite{B} is much more complete, establishing the positivity of the curvature in the sense of Nakano but we will not need this 
strong version here. We will discuss next a generalization of this result.
\medskip

\begin{notations} We consider the following context.
Let $p:X\to Y$ be an algebraic fiber space, and let $(L, h_L)$ be 
a singular Hermitian line bundle, whose curvature current is semi-positive. 
In this set-up, the direct image $\cE$ defined in \eqref{msri2} may not be a 
vector bundle anymore, but nevertheless $\cE$ is a torsion-free coherent sheaf.
Let 
\begin{equation}\label{msri7}
Y_{\cE}\subset Y
\end{equation}
be the largest subset of $Y$ for which the restriction of $\cE$ to $Y_{\cE}$ is locally  free. We note that $\codim (Y\sm Y_{\cE})\geq 2$, since $\cE$ is torsion-free. 

\noindent Our immediate goal in what follows will be to explain the construction of the metric $h_{\cE}$
in this more general setting.


We consider the maximal Zariski open subset $Y_0\subset Y$ such that the induced map
\begin{equation}\label{msri5}
p:X_0\to Y_0
\end{equation}
is a submersion, where $X_0:= p^{-1}(Y_0)$, and let $\Dl:=Y\sm Y_{0}$ be the complementary analytic set.
 
We denote by
$X_y$ the scheme theoretic fiber of $y\in Y$, and we consider
$\displaystyle L_y:=L|_{X_y}$ and $h_{L,y}:=h_L|_{X_{y}}$ the corresponding 
restrictions; we remark that it can happen that we have $h_{L, y}\equiv +\infty $. 
Let 
\begin{equation}
\cI(h_L)\subset \cO_X
\end{equation} 
be the multiplier ideal associated to the metric $h_L$. 

\noindent We consider the following set
\begin{equation}
	Y_{h_L} =\{y\in Y_0; h_{y}\not\equiv +\infty \}. 
\end{equation}
We remark that $\displaystyle Y_{h_L}$ is Zariski dense in $Y$, but it 
may not be Zariski open.
The complement $Y \sm Y_{h_L}$ is a pluripolar set and hence can be 
quite different from algebraic/analytic objects.
For example, a pluripolar set may not be closed in Hausdorff topology, and can be Zariski dense. 
\end{notations}

\subsubsection{Local expression of the metric}\label{loc_ex}
We give here an explicit formulation of the {\it canonical $L^{2}$-metric} 
$h_{\cE}$ on $\cE:= p_{\star}(K_{X/Y}+L)|_{Y_{0}}$.
For this purpose, we may suppose that $Y$ itself is a coordinate neighborhood.
It is crucial 
to understand the local expression of this metric,
in order to derive later its extension properties.
\medskip

\noindent Let $\eta \in H^{0}(Y, K_{Y})$ be a nowhere vanishing section, trivializing the canonical bundle of $Y$; in particular we have $K_{Y}=\CO_{Y}\eta$.
We recall that we have 
$$H^{0}(Y,p_{\star}(K_{X/Y}+L))=H^{0}(Y, \mathcal Hom\, (K_{Y},p_{\star}(K_{X}+L))),$$ and therefore every section $u \in H^{0}(Y,p_{\star}(K_{X/Y}+L))$ corresponds 
to a map
\begin{equation}
u:K_{Y} \to p_{\star}(K_{X}+L)
\end{equation}
which is an $\CO_{Y}$-homomorphism.
We still use the same symbol $u$ for the induced homomorphism 
\begin{equation}
u : H^{0}(Y,K_{Y}) \to H^{0}(Y,p_{\star}(K_{X}+L))=H^{0}(X,K_{X}+L)
\end{equation}
and we write $u(\eta) \in H^{0}(X,K_{X}+L)$.

Let $\{U_{\lam} \}_{\lam}$ be a local coordinate system of $X$.
Regarding $u(\eta)$ as an $L$-valued top-degree holomorphic form on each $U_{\lam}$,
there exists $\sg_{u\lam} \in  H^{0}(U_{\lam}\sm p^{-1}(\Dl), K_X+ L)$
such that we have  
\begin{equation}
	u(\eta)=\sg_{u\lam}\wed p^{\star}\eta
\end{equation}
on $U_{\lam}\sm p^{-1}(\Dl)$ i.e.\ we can ``divide" $u(\eta)$ by 
$p^{\star}\eta$ where $p^{\star}\eta$ has no zeros.

It is important to remark that the 
choice of $\sg_{u\lam}$ is not unique (the ambiguity lies in the image of $\Omega_{X}^{n-1}\ot p^{\star}\Omega_{Y}^{1}$). However 
the restriction 
\begin{equation}
\sg_{u\lam}|_{X_{y}} \in H^{0}(U_{\lam}\cap X_{y}, K_{X_{y}}+ L_{y})
\end{equation} 
on each smooth fiber $X_{y}$ --recall that $(y\in Y_{0}$)-- 
is unique and independent of the local frame $\eta$.

In conclusion the collection $\{\sg_{u\lam} \}_{\lam}$, resp.\ $\{\sg_{u\lam}|_{X_{y}} \}_{\lam}$, glue together as a global section 
\begin{equation}\label{msri10}
	\sg_{u} \in H^{0}(X \sm f^{-1}(\Dl), K_{X/Y}+ L), 
	\quad \sg_{uy} \in H^{0}(X_{y}, K_{X_{y}}+ L_{y})
\end{equation}
respectively; moreover, the latter is the restriction of the former  
$\sg_{uy}= \sg_{u}|_{X_{y}}$ for $y \in Y_{0}$. As we have already mentioned, $\sigma_u$ is not
uniquely defined, but its restriction to fibers is; thus it 
can be thought as ``representative'' of the section $u$.  
\smallskip

\noindent Then the expression of the 
canonical $L^{2}$-metric $h_{\cE}$ is given as follows.
Let
$u, v \in H^{0}(Y,p_{\star}(K_{X/Y}+L))$ be two local section at 
$y \in Y_{0}=Y\sm \Dl$; then we define
\begin{equation}\label{msri11}
	\langle u,v\rangle_y: = 
\int_{X_{y}} c_{n} \sg_{u}|_{X_{y}} \wed \ol \sg_{v}|_{X_{y}} e^{-\varphi_{L, y}} .
\end{equation}
We remark that the coefficients of the metric $h_{\cE}$ are indeed 
measurable functions, and that the scalar product \eqref{msri11} above is definite
positive.

\begin{remark} The convergence of the quantity \eqref{msri11} depends of course on 
the singularities of $h_{L, y}$. 
The metric $\displaystyle h_{\cE}$ is automatically $+\infty$ 
on the set $\displaystyle Y_0\sm Y_{h_L}$. 
We remark that $\displaystyle h_{\cE}$ is only defined on $Y_0$, but even so, it may happen that $(\cE|_{Y_0}, h_{\cE})$ \emph{is not} a singular Hermitian sheaf 
according to the definition \ref{sHm}. Indeed, it may happen that we have 
\begin{equation}
\det (h_{\cE})\equiv \infty
\end{equation}
on $Y_0$, because of the singularities of the restriction $h_L|_{X_y}$. 
\end{remark}
\smallskip

\noindent We will nevertheless exhibit next a general 
setup in which $h_{\cE}$ 
is a singular Hermitian metric on 
on the 
direct image sheaf $p_{\star}(K_{X/Y}+ L)$ 
in the sense we have defined in \ref{sHm}.
We remark that in general, the singularities 
of the metric we construct are unavoidable.

\begin{notations}\label{set on Y}
The following subsets of $Y$ will be needed in order to clarify the 
expression of the metric $h_{\cE}$.
We recall that $Y_0\subset Y$ is the set of regular values of $p$, and that 
$Y_h=\{ y\in Y_0;\ h_{L, y}=h_L|_{X_{y}}\not\equiv +\infty\}$.
\begin{equation*} 
\begin{aligned}
Y_{h,\rm ext} & :=\{y\in Y_{h};\ 
		H^{0}(X_{y},(K_{X_{y}}+L_{y})\ot \CI(h_y)) = H^{0}(X_{y},K_{X_{y}}+L_{y}) \}, 
\\
Y_{\rm ext} & := \{y \in Y_{0};\ 
		h^0(X_y,K_{X_y}+L_y) \text{ equals to the rank of } p_{\star}(K_{X/Y}+L) \}.
\\
Y_{\cE} & := \text {the largest Zariski open subset such that } p_\star(K_{X/Y}+L)|_{Y_{\cE}} \text{ is locally free. }	\end{aligned}
\end{equation*} 
\end{notations}

\noindent 
We remark that the set $Y_{\rm ext}$ and $Y_{\cE}$ 
are independent of the metric $h_L$. We also have the inclusion 
$Y_{\rm ext}\subset Y_{\cE}$. 
The next statement motivates an additional hypothesis we will make in a moment.

\begin{lem}\label{Y1h}\cite{mpst} The following assertions hold true.
\begin{enumerate}

\item[{\rm (1)}] We have the inclusion $Y_{h,\rm ext} \subset Y_{\rm ext}\cap Y_h$.
\smallskip

\item[{\rm (2)}] Let $y \in Y_{h,\rm ext}$.
Then the equalities 
\begin{equation}
p_{\star}(K_{X/Y}+L)_{y}
=H^{0}(X_{y},K_{X_{y}}+L_{y})
=H^{0}(X_{y},(K_{X_{y}}+L_{y})\ot \CI(h_y))
\end{equation}
hold. 
As a consequence, the sheaf 
$p_\star(K_{X/Y}+L)$ is locally free at $y$, and the natural inclusion 
\begin{equation}
p_{\star}((K_{X/Y}+L)\ot \CI(h)) \subset p_{\star}(K_{X/Y}+L)
\end{equation} 
is isomorphic at $y$.

\smallskip

\item[{\rm (3)}] If the natural inclusion $p_{\star}((K_{X/Y}+L)\ot \CI(h)) \subset p_{\star}(K_{X/Y}+L)$ is generically isomorphic, then $Y_{h,\rm ext}$ is not empty and 
$Y\sm Y_{h,\rm ext}$ has measure zero. 
\smallskip

\end{enumerate}
In conclusion, the sheaf $\displaystyle p_{\star}(K_{X/Y}+L)|_{Y_{\rm ext}}$ 
can be endowed with the canonical metric 
$\displaystyle \{h_{\cE,y}\}_{y\in Y_{\rm ext}}$. We note that 
we have $\det h_{\cE, y}=+\infty$ if $y \in Y_{\rm ext}\sm Y_{h,\rm ext}$.
\end{lem}

\begin{proof} The first point (1) is a consequence of the arguments 
we provide for at for (2), as follows. 
\smallskip
 
\noindent (2)
By Ohsawa-Takegoshi, every section $u\in H^0(X_y,K_{X_y}+L_y)$ admits an extension 
$$\wtil u \in  H^{0}(X_W,(K_{X/Y}+L)\ot \CI(h))$$ 
defined on some neighborhood $X_W$ of $X_y$.
In particular the natural induced map 
$$p_{\star}\big((K_{X/Y}+L)\ot \CI(h)\big)_{y} \to H^0(X_y,K_{X_y}+L_y)$$ is surjective
(we note that the left-hand side direct image above is contained in 
$p_{\star}(K_{X/Y}+L)_{y}$).
The cohomology base change theorem implies that 
$p_\star(K_{X/Y}+L)$ is locally free at $y$, and 
as a consequence the natural inclusion 
$p_{\star}((K_{X/Y}+L)\ot \CI(h)) \subset p_{\star}(K_{X/Y}+L)$ is isomorphic at
$y$.
This argument explains also the point (1).
\smallskip
 
\noindent (3) By hypothesis, 
there exists a non-empty Zariski open subset $W\subset Y$ such that the inclusion map
$$p_{\star}\left((K_{X/Y}+L)\ot \CI(h)\right) \subset p_{\star}(K_{X/Y}+L)$$ 
is isomorphic on $W$, and such that 
\begin{equation}
p_{\star}\left((K_{X/Y}+L)\ot \CI(h)\right)_{y} = H^{0}(X_{y},(K_{X_{y}}+L_{y})\ot \CI(h)\cdot \CO_{X_{y}})
\end{equation}  
as well as
\begin{equation}
p_{\star}(K_{X/Y}+L)_{y}= H^{0}(X_{y},K_{X_{y}}+L_{y})
\end{equation}
for any $y\in W$.
Since $\CI(h)\cdot\CO_{X_{y}}=\CI(h_{y})$ for almost all\,$y\in Y_{0}$ in general by \ref{restriction} below, our assertion follows.
\smallskip
 
\noindent This last part of the lemma follows directly from (1)--(3).
\end{proof}
\medskip

\noindent The following fact was used in the proof of Lemma \ref{Y1h}.

\vskip2mm

\begin{remark}\label{restriction} 
As a consequence of Ohsawa-Takegoshi extension theorem, we have the inclusion $\CI(h_{y}) \subset \CI(h) \cdot \CO_{X_{y}}$ for any $y\in Y_{0}$. 
The next set we will be interested in would be $y\in Y_{0}$ such that 
$$
	\CI(h_{y}) = \CI(h) \cdot \CO_{X_{y}}.
$$
In the algebraic case, this holds for any $y$ in a Zariski open subset (\cite[9.5.35]{PAG}).
We show here that this equality holds on a set of full measure on $Y$ (which in general is not Zariski open).

To this end, it is enough to show that $\CI(h) \cdot \CO_{X_{y}} \subset \CI(h_{y})$ 
for almost all \,$y\in Y$.
We may suppose that $Y$ is a small coordinate neighborhood.
Let $U\subset X$ be a local coordinate set, which is isomorphic to a polydisk, such that $f|_{U}$ is 
(conjugate to) the projection to a sub-polydisk.
We may assume that $\CI(h)|_{U}$ is generated by a finite number of holomorphic functions $s_{1},\ldots, s_{k}\in H^{0}(U,\CO_{X})$, in particular $|s_{i}|^{2}h \in L^{1}_{loc}(U)$.
As these $s_{i}|_{X_{y}}$ generate $\CI(h) \cdot \CO_{X_{y}\cap U}$, it is enough to show that each $s_{i}|_{X_{y}} \in \CI(h_{y})|_{X_{y}\cap U}$ for almost all \,$y\in p(U)$.
By Fubini theorem, $|s_{i}|_{X_{y}}|^{2}h_{y} \in L^{1}_{loc}(X_{y}\cap U)$ for a dense 
set of\,$y\in f(U)$, i.e., $s_{i}|_{X_{y}} \in \CI(h_{y})|_{X_{y}\cap U}$ for almost all \,$y\in p(U)$.
\smallskip
\end{remark}
\medskip

\noindent In view of Lemma \ref{Y1h} we see that the following hypothesis 
is natural.

\begin{ass}\label{ass_gen}
We assume that the inclusion
\begin{equation}\label{msri12}
p_{\star}\big((K_{X/Y}+L)\ot \CI(h)\big) \subset p_{\star}(K_{X/Y}+L)
\end{equation}
is generically isomorphic. We remark that for this to hold it is not necessary that 
the multiplier ideal sheaf $\cI(h_L)$ equals $\cO_X$.
\end{ass}
\smallskip

\noindent Thus under the assumption \eqref{msri12} the restriction 
\begin{equation}\label{msri13}
\big(\cE|_{Y_{\rm ext}}, h_{\cE}\}\big) 
\end{equation}
is a singular Hermitian vector bundle in the sense of Definition \ref{sHm}. We note however
that the set $Y_{\cE}\setminus Y_{\rm ext}$ could be quite ``large'' in the sense that 
it may contain a codimension one algebraic set. Next, we show that 
$h_{\cE}$ is semi-positively curved and it admits an extension to $Y_{\cE}$.


\medskip


\subsection{Positivity and extension properties of the $L^2$ metric}

\noindent We recall next the following result.

\begin{thm}\cite[3.5]{BPDuke}, \cite{mpst}\label{bp35}
Let $p:X\to Y$ be a smooth algebraic fiber space and let $(L,h_L)$ be a positively curved line bundle, such that \eqref{msri12} holds.
Suppose moreover that $\cE:= p_{\star}(K_{X/Y}+L)$ is locally free. The following assertions hold true.
\begin{enumerate}

\item[{\rm (1)}]
The singular Hermitian vector bundle
$(\cE, h_{\cE})$ with positively curved. 
\smallskip

\item[{\rm (2)}] The following {\rm base change property} holds on $Y_{\rm ext}$: 
the restriction of the metric $\displaystyle h_{\cE}$ 
on the 
 fiber $p_{\star}(K_{X/Y}+L)_{y}$ at $y\in Y_{\rm ext}$, 
is given by the formula {\rm \eqref{msri11}}.
\end{enumerate}
\end{thm}

\noindent The part (2) of the result above 
is slightly more informative than the original one in \cite[3.5]{BPDuke}.
We stress on the fact that the statement above is implicit in \cite{BPDuke}, it is not mentioned explicitly.

\medskip

\noindent The main result in \cite{mpst} reads as follows.

\begin{thm}\label{mpst} 
Let $p:X\to Y$ be an algebraic fiber space and let $(L,h_L)$ be a positively curved line bundle, such that \eqref{msri12} holds.
Then the canonical $L^{2}$-metric $h_{\cE}$ on $\cE|_{Y_{\rm ext}}$ 
extends as a positively curved 
singular Hermitian metric $\wtil h_{\cE}$ on the torsion free sheaf $\cE$.
\end{thm}

\noindent In the rest of this section we will highlight  
the key steps in the proof of \ref{mpst} as in \cite{mpst}. 
One can see that to some extent, our arguments represent a 
generalization of 
the work of T.\ Fujita in \cite{Ft}.
 
We start by making some standard reductions and by fixing some notations. 
In order to lighten the writing, let
\begin{equation}
g:= h_{\cE}
\end{equation}
be the $L^2$ metric defined on $Y_{\rm ext}$ as explained above.

We have $\codim (Y\sm Y_{\cE})\ge 2$ so it 
is enough to show that $g$ extends to a singular Hermitian metric on 
$\displaystyle \cE|_{Y_{\cE}}$ with
positive curvature, given the definition 
of singular Hermitian metric for torsion-free coherent sheaves (see \ref{dual}).
Thus we may assume from the start that $\cE$ 
is locally free, in other words we have $Y_{\cE}= Y$.
We also note that the extension of $g$ is a local matter on $Y$.
Moreover, we can freely restrict ourselves on a Zariski 
open subset $Y'' \subset Y$ with $\codim (Y\sm Y'')\ge 2$, since the extended metric is unique.
\smallskip

\noindent In conclusion, it would be enough to show that $g$ extends across the
codimension one components of the set $Y_{\cE}\sm Y_{\rm ext}$. Let $\Sigma$ be such a component. The singular points of $\Sigma$ have codimension at least two in $Y$, so we will simply ignore them. 
We will assume that the base $Y$ is a unit polydisk in $\bC^m$ 
with coordinates $t = (t_1, \ldots, t_m)$, such that 
\begin{equation}
\Sigma= (t_m= 0).
\end{equation}
Let $dt = dt_1 \wed \ldots \wed dt_m \in H^0(Y, K_Y)$ be the frame of $K_{Y}$
corresponding to the $t$--coordinates. We write next the $p$-inverse image of 
$\Sigma$ as follows
\begin{equation}
p^\star(\Sigma)= \sum_{j\in J_v}b^jZ_j+ \sum_{j\in J_h}b^jZ_j 
\end{equation}
where the hypersurfaces $(Z_j)\subset X$ corresponding to indexes $J_v$ project into a proper analytic subset of $\Sigma$, and $p(Z_j)= \Sigma$ for all $j\in J_{h}$. Moreover, we can assume that $(Z_j)$ have simple normal crossings, and that the induced map
$\Supp p^*\Sigma \to \Sigma$ is relative normal crossing 
: this can be achieved by a
birational transform of $X$ (cf.\ \cite{mpst}, Remark 3.2.4).
\smallskip

\noindent All in all, it is enough to work in the following setup.   

\smallskip
\noindent
(1) We consider a general point $\displaystyle y_0\in \Sigma\sm \cup_{j\in J_v}p(Z_j)$,
and a open set $y_0\in \Omega$ such that $\cE|_{\Omega}$ is trivialized by the sections
$u_1,\dots, u_r$; with respect to the $t$ coordinates above, we have $y_0= 0$.
\smallskip

\noindent
(2) Locally near every point $x \in X$, there exists a local coordinate 
\begin{equation}
(U; z = (z_1, \ldots, z_{n+m}))
\end{equation} 
such that $p|_U$ is given by $t_1 = z_{n+1}, \ldots, t_{m-1} = z_{n+m-1}$, $t_m = z_{n+m}^{b_{n+m}} \prod_{j=1}^{n} z_j^{b_j}$ with non-negative integers $b_j$ and $b_{n+m}$.
\smallskip

\noindent We will show next that the  positively curved (cf.\ \ref{bp35}) 
canonical $L^{2}$-metric $g$ on $\cE$ 
defined over $Y_{0}$ extends locally near $y_0$; according to the discussion above, this would be enough to conclude.

Let $(\xi_j)_{j=1,\dots, r}$ be the base of $\cE^\star$ induced by 
$(u_j)$; by definition, this means that $\displaystyle \xi_j(u_k)= \delta_{jk}$. 
We already know that the function
\begin{equation}\label{msri15}
t\to \log \Vert \xi_j(t)\Vert_{g^\star}
\end{equation}
defined on $\Omega\setminus (t_m= 0)$ is psh for any $j=1,\dots, r$. 
If we are able to show that 
\begin{equation}\label{msri16}
\sup_{t\in \Omega\sm \Sigma}\Vert \xi_j(t)\Vert_{g^\star}< \infty
\end{equation}
then by Hartogs theorem the function \eqref{msri15} admits a unique psh extension 
to $\Omega$, which is exactly what we have to prove.
\smallskip

\noindent We will argue by contradiction: assume that \eqref{msri16} does not holds
say for $j=1$.
Then we obtain a sequence of points 
$\displaystyle y_k\in \Omega\setminus (t_m= 0)$ such that 
\begin{equation}
y_k\to y_\infty\in (t_m= 0)
\end{equation}
together with a sequence of sections $\wt v_k$ of $\cE$ written 
as
\begin{equation}
\wt v_k= \sum_{p=1}^r \mu^p_k u_p
\end{equation}
such that 
\begin{equation}\label{msri 17}
\Vert \wt v_k(y_k)\Vert_{g}= 1,\quad |\mu^1_k|\to \infty.
\end{equation}
Let $v_k:= \frac{1}{\Vert \mu_k\Vert}\wt v_k$ be the rescaling of $\wt v_k$ by the 
norm of its coefficients $\mu_k\in \bC^r$. This new sequence of sections 
\begin{equation}
v_k= \sum_{p=1}^r \lambda^p_k u_p
\end{equation}
has the following properties. 
\begin{enumerate}
\smallskip

\item[(i)] The sequence of coefficients
$\lambda_k$ converges to $\lambda_\infty$ as $k\to \infty$ so we obtain a limit
\begin{equation}\label{msri18}
v_k\to v_\infty: = \sum_{p=1}^r \lambda^p_\infty u_p
\end{equation}
and $\Vert \lambda_\infty\Vert= 1$, so that $\lambda_\infty$ belongs to the 
unit sphere of the Euclidean space $\bC^r$.
\smallskip

\item[(ii)] We have $\displaystyle 
\Vert v_k(y_k)\Vert_{g}\to 0$ as $k\to \infty$.
\end{enumerate}

\noindent Intuitively at least ``we are done'': the section $v_{\infty}$ is non-zero as 
element of the stalk $\displaystyle \cE_{y_\infty}$ by the point (i), whereas 
(ii) seems to say that the norm of $v_\infty$ at $y_\infty$ is zero. Except that there are basically two obstacles to overcome. The first one is that the fiber $p^{-1}(y_\infty)$
is singular (reducible, non-reduced...) as we see from the point (2) above, so it is not so clear what it means that \emph{$v_\infty$ is non zero as element of the stalk $\displaystyle \cE_{y_\infty}$}. The second one is that the norm of $v_\infty$ at $y_\infty$ 
is not defined yet --in point of fact, this is what we are after...
\smallskip

\noindent In the paper \cite{mpst}, the difficulties mentioned above are overcome in
two steps, as follows.
\smallskip

\noindent $\bullet$ If $p$ is \emph{semi-stable} in codimension one, meaning that the coefficients $\displaystyle (b_j)_{j=1,\dots, n, n+m }$ are either zero or one, then one first shows that there exists a component of $p^{-1}(y_\infty)$ such that the top degree form
$u_\infty\wedge p^\star(dt)$ is not identically zero when restricted 
to this component. Then a continuity argument, combined with the fact that on $\Omega\setminus (t_m= 0)$ the expression of the metric $g$ is explicit, one is able to show that 
(i) and (ii) cannot hold simultaneously.
\smallskip

\noindent $\bullet$ The general case (i.e. arbitrary coefficients $b_j$) is addressed by using the previous bullet (the semi-stable case), together with the weak 
semi-stable reduction theorem, cf.\ \cite{KKMS}, \cite{Nbook} combined with 
a fundamental result of Viehweg. Note that the properties of positively curved 
singular Hermitian metric are playing a role in this part of the proof as well.
\medskip

\noindent We will not reproduce here completely the arguments in \cite{mpst}. 
Instead, we will sketch the proof of a particular case of the first bullet above and explain the outline of the proof of the second one, as follows.
We assume that the base $Y$ is one-dimensional, and that the local expression of $p$ is
\begin{equation}
(z_1,\dots, z_n, z_{n+1})\to z_1z_{n+1}
\end{equation}
so basically we assume that we only have 
a simple normal crossing fiber as singularity at the origin. 
\smallskip

\noindent We have the following statement, cf.\ \cite{mpst}, Lemma 3.3.8 and beginning of proof of Lemma 3.3.12.

\begin{lemma}\cite{mpst}\label{restrict}
Let $\rho_\infty:= v_\infty(dt)$ be the holomorphic 
section of $\displaystyle K_X+ L|_{p^{-1}(\Omega)}$ 
corresponding to $v_\infty$. Then the vanishing order of $\rho_\infty$ along one of the components of the snc 
divisor
$B$ given locally by $(z_1z_{n+1}=0)\subset X$ is equal to zero..
\end{lemma}

\begin{proof}
We assume the contrary, i.e. the local holomorphic form $\rho_\infty$ 
vanishes along $B$. Then we claim that the quotient
\begin{equation}\label{msri20}
\eta:= \frac{v_{\infty}}{p^\star(t)}
\end{equation}
is a \emph{holomorphic} section of $\cE|_{\Omega}$. Indeed, if we are able to do so,
then we infer that $v_\infty\in m_0\cE$, in other words, it corresponds to zero when restricted to the stalk $\cE_{0}$ (here we denote by $m_0$ the maximal ideal of $0\in Y$). 
On the other hand, we know a-priori that this 
is not the case.

In order to justify the fact that $\eta$ is holomorphic, we have to show that it 
induces a $\cO_Y$-morphism
\begin{equation}
K_Y\to p_\star(K_X+ L). 
\end{equation}
This is however clear: $\displaystyle \eta\wedge p^\star(dt)= \frac{\rho_\infty}
{p^\star(t)}$ which is holomorphic thanks to the fact that the vanishing order of
$\rho_\infty$ is large enough.     
\end{proof}
\medskip

\noindent Let $\sigma_j$ be the $n$-form defined on the coordinate open set $V$ corresponding to $u_j$, cf.\ \ref{loc_ex}. Then we have 
\begin{equation}
\sum \lambda^j_\infty \sigma_j\wedge p^\star(dt)= \rho_\infty 
\end{equation}
by the definition of $v_\infty$. We assume that $\rho_\infty$ is not identically zero 
when restricted to $z_{n+1}= 0$, cf.\ Lemma \ref{restrict} above.
For each $j=1,\dots, r$ we write next
\begin{equation}
\sigma_j\wedge p^\star(dt)= \sum_{p\geq 0}a_{jp}(z^\prime)z_{n+1}^p
\end{equation}
where $z^\prime:= (z_1,\dots, z_n)$ and thus the holomorphic function
\begin{equation}\label{msri21}
\sum_j\lambda^j_\infty a_{j0}(z^\prime)
\end{equation}
is not identically zero.  
By eventually shrinking the set $V$, it follows that we have
\begin{equation}\label{msri22}
\inf_V\left|{\sum_{p\geq 0}\sum_j\lambda^j_\infty a_{jp}(z^\prime)z_{n+1}^p}\right|> 0.
\end{equation}
If $k\gg 0$, then we equally have
\begin{equation}\label{msri23}
\inf_V\left|{\sum_{p\geq 0}\sum_j\lambda^j_k a_{jp}(z^\prime)z_{n+1}^p}\right|\geq \ep_0> 0
\end{equation}
and then we easily derive a \emph{uniform} lower bound for the norm of $v_k$ at $y_k$ 
as follows.

In terms of local forms $(\sigma_j)$ the inequality \eqref{msri23} can be rewritten as
\begin{equation}\label{msri24} 
\inf_V\left|\frac{\sum_j\lambda^j_k\sigma_j\wedge \ol{\sum_j\lambda^j_k\sigma_j}}
{dz^\prime\wedge d\ol z^\prime}\right|\geq \ep_0> 0
\end{equation}
where $dz^\prime:= dz_1\wedge\dots\wedge dz_n$. We remark that \eqref{msri24}
holds true as soon as $k\gg 0$ is large enough. But then we have
\begin{equation}\label{mess1}
\Vert v_k(t)\Vert^2\geq \int_{X_t\cap V}c_n
\sum_j\lambda^j_k\sigma_j\wedge \ol{\sum_j\lambda^j_k\sigma_j}e^{-\varphi_L}\geq 
\ep_0\Vol(V\cap X_t)
\end{equation} 
for any $t\neq 0$,
and the sketch of the proof of the first bullet is finished, since \eqref{mess1}
clearly contradicts the point (ii) above. 

\smallskip

\noindent The main steps of the remaining part of the 
argument are as follows. By the semi-stable reduction theorem, there exists a finite map
\begin{equation}
\tau: \wt Y\to Y
\end{equation}
and an induced map $\wt p: \wt X\to \wt Y$, where $\wt X$ is the desingularization of the 
(main component of the) fibered product $\wt Y\times_Y\!X$ such that the 
coefficients $b_j$ above corresponding to $\wt p$ are equal to 0 or 1. By the results of 
E. Viehweg \cite{Vi1}, there exists a morphism of sheaves
\begin{equation}
\wt \cE:= p_{\star}(K_{\wt X/\wt Y}+ \wt L)\to \tau^\star\left(p_{\star}(K_{X/Y}+ L)\right)
\end{equation} 
which is generically \emph{isomorphic}; actually, it is an isometry on $Y\sm 0$, as we 
see directly from the definition of the $L^2$ metric. Now, the metric 
$\displaystyle h_{\wh \cE}$ extends thanks to the fact that $\wt p$ is semi-stable in codimension one;
the proof is completed thanks to Lemma \ref{subquot2}, point (3).  \qed
\smallskip

\begin{rem}
It may be possible to show that the sheaf
\begin{equation}
p_{\star}\left((K_{X/Y}+L)\ot \CI(h)\right)
\end{equation} 
is positively curved without assuming that the inclusion
$$p_{\star}\left((K_{X/Y}+L)\ot \CI(h)\right) \subset p_{\star}(K_{X/Y}+L)$$ 
is generically isomorphic.
However there are some technical difficulties to overcome.
To explain this, suppose $Y$ is a disk, and $X_{0}=\sum m_{i}X_{i}$ is the singular fiber over $0\in Y$. 
There are difficulties to compare the ideal $\CI_{X_{0}}, \CI(h)$, and $\text{div} (u)$ for $u \in H^{0}(X,(K_{X/Y}+L)\ot \CI(h))$.
Embedded components of $\CO_{X}/\CI(h)$ are difficult to handle in general.
A more serious trouble may arise because of the fact that the metric 
$h_L$ may have non-algebraic singularities, so we cannot reduce to the normal crossing situation even after a modification of the manifold.
\end{rem}
\medskip


\subsection{Direct images of pluricanonical bundles}

\noindent In the article \cite{mpst} it is established that 
under some reasonable assumptions which will be made precise 
in a moment, the sheaf  
\begin{equation}\label{msri24}
\cE_m:= p_\star\left(m(K_{X/Y}+ L)\right),
\end{equation}
is positively curved when endowed with a natural metric, 
for any $m\geq 1$. As usual, 
here $p:X\to Y$ is an algebraic fiber space and $(L, h_L)$ is a positively curved 
singular Hermitian $\bQ$-line bundle, such that $mL$ is Cartier. The idea is to reduce ourselves to the 
case $m= 1$, as follows. We write
\begin{equation}
m(K_{X/Y}+ L)= K_{X/Y}+ L_m
\end{equation}
where $L_m= L+ (m-1)(K_{X/Y}+ L)$, and then we clearly have
\begin{equation} 
\cE_m= p_\star(K_{X/Y}+ L_m).
\end{equation}
In the next subsections, we will see that in some cases one can construct a 
metric  $h_m$ on $L_m$ such that the assumption \ref{ass_gen} is satisfied. 

\smallskip  
 \subsubsection{The relative Bergman metric}
 
 \noindent The general set-up for the current section is as follows. 
Let $X$ and $Y$ be two projective manifolds, which are assumed to be non-singular.
 Let $p: X\to Y$ be a surjective map, and let $(L, h_L)\to X$ be a 
line bundle endowed with a Hermitian  
 metric $h_L$. We assume that we have
 \begin{equation}\label{equa10}
 \Theta_{h_L}(L)\geq 0
 \end{equation}
 in the sense of currents on $X$.
 
 \noindent 
 In this context we recall  the construction of the Bergman 
metric $\displaystyle e^{-\varphi_{X/Y}}$ on the bundle $K_{X/Y}+ L$; we refer to \cite{BPDuke} for further details. 

\noindent Let $Y_0$ be a Zariski open subset of $Y$ such that $p$ is smooth over $Y_0$, and for every $y\in Y_0$, the fiber $X_y$ satisfies 
$h^0 (X_y, K_{X/Y}\otimes L\otimes \mathcal{I} (h_L |_{X_y})) = \rank p_* (K_{X/Y}\otimes L\otimes \mathcal{I} (h_L))$.
Let $X^0$ be the $p$-inverse image of $Y_0$ and let $x_0\in X^0$ be an arbitrary point; let $z^1,\dots, z^{n+m}$ be local coordinates centered at $x_0$, 
and let $t^1,\dots , t^m$ be a coordinate
centered at $y_0:= p(x_0)$. We consider as well a trivialization of $L$ near $x_0$. 
With this choice of local coordinates, we have a local trivialization of the tangent bundles of $X$ and $Y$ respectively, and hence of the (twisted) relative canonical bundle. 

\noindent The local weight of the metric $\displaystyle e^{-\varphi_{X/Y}}$ with respect to this is given by the equality

\begin{equation}\label{relative}
e^{\varphi_{X/Y}(x_0)}= \sup_{\Vert u\Vert_{y_0}\leq 1} |F_u (x_0)|^2
\end{equation}
where the notations are as follows: $u$ is a section of $\displaystyle K_{X_{y_0}}+ L|_{X_{y_0}}$, 
and $F_u$ corresponds to the local expression of $u\wedge p^\star dt$, i.e. the coefficient of $dz^1\wedge \dots \wedge dz^{n+m}$. 
The norm which appears in the definition \eqref{relative} is obtained by the fiber integral
\begin{equation}\label{equa121}
\Vert u\Vert_{y_0} ^2:= \int_{X_{y_0}} |u|^2e^{-\varphi_L}.
\end{equation}
\medskip

\noindent An equivalent way of defining \eqref{relative} is via an orthonormal basis, say $u_1,
\dots , u_k$ of sections of $\displaystyle K_{X_{y_0}}+ L|_{X_{y_0}}$. Then we see that
\begin{equation}\label{on}
e^{\varphi_{X/Y}(x_0)}= \sum_{j=1}^N |F_j(x_0)|^2
\end{equation}
where $F_j$ are the functions corresponding to $u_j$. 

\noindent The Bergman metric $\displaystyle h_{X/Y}= e^{-\varphi_{X/Y}}$ can also be introduced in an
 intrinsic manner as follows. Let $\xi$ be a vector in the fiber over $x_0$ of the dual bundle 
 $\displaystyle -(K_{X/Y}+ L)_{x_0}$.
 The expression
 \begin{equation}\label{equa11}
 \vert \xi\vert^2= \sup_{\Vert u\Vert_{y_0}\leq 1} |\langle \xi, u\rangle |^2
 \end{equation}
defines a metric on the dual bundle, whose local weight is precisely $\varphi_{X/Y}$. 
\medskip

\noindent As we see from \eqref{on}, the restriction of the metric 
$e^{\varphi_{X/Y}}$ to the fiber $X_{y_0}$ coincides with the metric induced by any \emph{orthonormal basis}
of the space of holomorphic sections of $\displaystyle K_{X_{y_0}}+ L|_{X_{y_0}}$. Hence the variation from one fiber to another is in general a $\mathcal C^\infty$ operation, since the said orthonormalization process is involved. Thus it is a remarkable fact that this metric has positive curvature in the sense of currents on $X$.

\begin{theorem} {\rm (\cite[Thm 0.1]{BPDuke})}\label{rel1} The curvature of the metric $h_{X/Y}$ on the twisted relative canonical bundle $\displaystyle K_{X/Y}+ L|_{X^0}$ is positive in the sense of currents. Moreover, the local weights $\varphi_{X/Y}$ are uniformly bounded from above on $X^0$, so they admit a unique extension as psh functions.
\end{theorem}
\medskip

\noindent We will not reproduce here the proof of the preceding result. The first complete argument was given in \cite{BPDuke}, in which Theorem \ref{rel1} is obtained as consequence of the Griffiths positivity of $p_\star(K_{X/Y}+ L)$, combined with 
a regularization procedure. In \cite{mpst}, section 4 
(and references therein), an alternative argument (originally due to H.~Tsuji) is presented, based on 
the version of OT theorem with optimal constant obtained by Blocki and Guan-Zhou, cf.\
\cite{blocki}, \cite{GZ}; see also \cite{JCao}.

\smallskip

\noindent The definition \ref{relative}, although not intrinsically formulated, is explicit 
enough so as to imply the following statement. Let $p:X\to Y$ be a dominant map, 
such that $X$ is K\"ahler; we denote by 
$\Delta$ the analytic set corresponding to the critical values of $p$, and we assume that the $p$-inverse image of $\Delta$ equals
\begin{equation}\label{equa221}
\sum_{i\in I} e_iW_i
\end{equation}
where $e^i$ are positive integers, and $W_i$ are reduced hypersurfaces of $X$. 
\medskip

\noindent The next statement can be seen as a metric version of the corresponding results due to 
Y. Kawamata in \cite{Ka98} and of F. Campana in \cite{Cam04}, respectively. We will not use it in what follows, but it is interesting to see that the singularities of the 
map $p$ are taken into account by the singularities of the Bergman 
metric $e^{-\varphi_{X/Y}}$.  
\begin{theorem}\label{nonreduce}
Let $\Theta_{X/Y}$ be the curvature current corresponding to the Bergman metric {\rm \ref{relative}}. Then we have 
\begin{equation}\label{equa222}
\Theta_{X/Y}\geq [\Sigma_p]:= \sum_{i\in I_h}(e_i-1)[W_i]
\end{equation}
in the sense of currents on $X$ where $I_h$ is the set of indexes $i\in I$ such that $p(W_i)$ is a divisor of $Y$. In other words,   
the current $\Theta_{X/Y}$ is singular along the multiple fibers of the map $p$.
\end{theorem}

\begin{proof}
Let $x_0\in W_1$ be a non-singular point of one of the sets appearing in \eqref{equa221}. 

We consider a coordinate set $\Omega$ containing the point $x_0$, and we fix the coordinates 
$\displaystyle (z_1,\dots , z_{n+m})$ on $\Omega$, such that 
$W_1\cap \Omega= (z_{n+ 1}= 0)$. The local structure of the map $p$ is as follows
\begin{equation}\label{eq1}
\big(z_1,\dots, z_{n+m})\to (z_{n+1},\dots, z_{n+ m-1}, z_{n+m}^{e_{n+m}}\big).
\end{equation}

Then we see that the intersection of the fibers of $p$ near $p(x_0)$ 
with $\Omega$ can be identified with the unit disk in $\bC^n$.  The upshot is that 
the normalization  \eqref{equa121} allows us to bound the absolute value of the restriction of the 
section $u$ which computes the Bergman metric near $x_0$. 

More precisely, let $u$ is a section of the $\displaystyle K_{X_{y_0}}+ L|_{X_{y_0}}$ as in \eqref{equa121}. We assume that we have $\|u\|_{X_{y_0}} ^2=1$, and 
by the construction of $F_u$, we have
$$\int_{X_{y_0} \cap \Omega} \frac{|F_u|^2}{|z_{n+m}|^{2(e_{n+ m}- 1)}} 
d\lambda(z^\prime) \leq \|u\|_{X_{y_0}} ^2 =1 $$
where $d\lambda(z^\prime)$ is the Lebesgue measure corresponding to the first $n$ variables
$z_1,\dots, z_n$.
Combining this with \eqref{relative}, we have thus
\begin{equation}\label{eq2002}
\varphi_{X/Y}(z)\leq (e_{n+ 1}-1)\log\vert z_{n+ 1}\vert^2+ \cO(1) ,
\end{equation}
and the proof is finished.
\end{proof}
\medskip


\bigskip

\noindent The construction of the metric $h_{X/Y}$ has a perfect pluricanonical analogue, as we recall next.
Let $u$ be a section of the bundle $\displaystyle m(K_{X_y}+ L)$, where $m\geq 1$ is a positive integer; here $L$ can even be a $\bQ$-line bundle, but in that case $m$ 
has to be divisible enough so that $mL$ is a genuine (Cartier) line bundle.
 
\noindent Then we define 
\begin{equation}\label{eq41}
\Vert u\Vert^{\frac{2}{m}}_y:= \int_{X_y}\vert u\vert^{\frac{2}{m}}e^{-\varphi_L},
\end{equation}
and the definition \eqref{equa11} generalizes immediately, as follows. Let $\xi$ be a vector in the fiber 
over $x$ of the dual bundle 
 $\displaystyle -m(K_{X/Y}+ L)_{x}$
 The we have
 \begin{equation}\label{eq44}
 \vert \xi\vert^2= \sup_{\Vert u\Vert_y\leq 1} |\langle \xi, u\rangle |^2.
 \end{equation}
We denote the 
resulting metric by $h^{(m)}_{X/Y}$.
\medskip

\noindent We recall next the analogue of Theorem \ref{rel1}, as follows.

\begin{theorem} {\rm (\cite[Thm 0.1]{BPDuke})}\label{rel2} The curvature of the metric $\displaystyle h_{X/Y}^{(m)}$ 
on the twisted relative pluricanonical bundle 
$\displaystyle m\left(K_{X/Y}+ L\right)|_{X^0}$ is positive in the sense of currents. Moreover, the local weights $\varphi_{X/Y}^{(m)}$ are uniformly bounded from 
above on $X^0$, so they admit a unique extension as psh functions.
\end{theorem}

\begin{remark}\label{msing} {\rm
If the map $p$ verifies the hypothesis of Theorem \ref{nonreduce}, then we infer that 
\begin{equation}
\Theta_{h_{X/Y}^{(m)}}\big(mK_{X/Y}+ mL\big)\geq m[\Sigma_p].
\end{equation}
The proof is absolutely the same as in Theorem \ref{nonreduce}: if the local structure of the map $p$ is as in \eqref{eq1}, then the 
$L^{2/m}$ normalization bound for the sections involved in the computation of the metric $h_{X/Y}^{(m)}$
imply that the local pointwise norm of these sections is bounded. The weights of the metric 
 $\displaystyle h_{X/Y}^{(m)}$ are given by the wedge product with $dt^{\otimes m}$, so the conclusion follows.
}
\end{remark}

\bigskip

\noindent As a consequence of Theorem \ref{rel2} and Theorem \ref{mpst} we obtain in 
\cite{mpst} the following result.

\begin{theorem}\label{plurican}
Let $p:X\to Y$ be an algebraic fiber space, and let $(L, h_L)$ be a positively curved 
Hermitian $\bQ$-line bundle. We assume that the multiplier ideal sheaf $\cI(h_L)$
associated to $h_L$ (which is well-defined, despite of the fact that $L$ is only a $\bQ$-bundle) is equal to $\cO_X$, and we endow the bundle
\begin{equation}
L_m:= L+ (m-1)(K_{X/Y}+ L)
\end{equation}
with the metric $\displaystyle h_m:= e^{-\vph_L-\frac{m-1}{m}\vph^{(m)}_{X/Y}}$. Then the 
resulting metric $h^{(m)}_{\cE}$ on the direct image sheaf $\cE_m$ is 
positively curved.
\end{theorem}

\begin{proof}
By hypothesis and Theorem \ref{rel2}, the metric $h_m$ defined as indicated above 
is positively curved. Hence we only have to check that the inclusion 
\begin{equation}
p_\star\left((K_{X/Y}+ L_m)\otimes \cI(h_m)\right)\subset \cE_m= p_\star(K_{X/Y}+ L_m)
\end{equation}
is generically isomorphic. This is easy: let $u$ be a local section of $\cE_m$, defined on $\Omega\subset Y_0$ (i.e. $\Omega$ is contained in the set of regular values of $p$). 
It would be enough to show that we have
\begin{equation}
\int_{p^{-1}(\Omega)}\vert u(dt)\vert^2e^{-\vph_{L_m}}< \infty.
\end{equation}
By definition of the metric $\vph^{(m)}_{X/Y}$ we have 
\begin{equation}\label{msri30}
\int_{p^{-1}(\Omega)}\vert u(dt)\vert^2e^{-\vph_{L_m}}\leq 
\int_{p^{-1}(\Omega)}\vert u(dt)\vert^{2/m}e^{-\vph_{L}},
\end{equation}
and the right hand side term in \eqref{msri30} is convergent, given that $\cI(h_L)= \cO_X$. 
\end{proof}

\medskip

\begin{rem}
Actually, the hypothesis ``$\cI(h_L)= \cO_X$'' can be relaxed. As we see from the 
proof above, Theorem \ref{plurican} would follow provided that
we have 
\begin{equation}
\int_{p^{-1}(\Omega)}\vert u(dt)\vert^{2/m}e^{-\vph_{L}}<\infty
\end{equation} 
for any local section $u$ of $\cE_m$ defined on some open set $\Omega$ contained in a Zariski 
dense subset of $X$.
\end{rem}
\medskip

\section{Further results}

\smallskip

\noindent We will present next a beautiful result recently obtained  
by S.~Takayama in \cite{taka2}. He has established 
an estimate concerning the singularities 
of the metric $h^{(m)}_\cE$ on $p_{\star}(mK_{X/Y})$ in the following context.

Let $p:X\to C$ be an algebraic fiber space over a curve (so that 
in particular $X$ and $C$ are non-singular). 
We assume that for some positive integer $m$ the direct image vector bundle
\begin{equation}\label{msri50}
\cE_m:= p_\star\left(mK_{X/Y}\right)
\end{equation}
is non-zero, of rank $r$. Let $y_1,\dots, y_s$ be the set of singular values of $p$, and let $C^\star:= C\sm\{y_1,\dots, y_s\}$. 

Let $u\in H^0(\Omega, \cE_m)$ be a local holomorphic section of \eqref{msri50}.
It is clear that the function
\begin{equation}\label{msri51}
\tau\to |u(\tau)|^2_{h^{(m)}_{\cE}}
\end{equation}
is locally bounded at each point of the intersection $\Omega\cap C^\star$. 
\smallskip

\noindent In the article \cite{taka2}, the author is obtaining an upper bound of
\eqref{msri51} near the singular points of $p$. Let $y_1$ be one of the singular points of $p$, and let $t$ be a local coordinate on the curve $C$ centered at $y_1$. 
We denote by $X_0$ the scheme-theoretic fiber of $p$ at $0$. Then $X_0$ is a divisor of $X$, 
and let 
\begin{equation}\label{msri52}
\nu_0:= \sup\left\{r> 0: \frac{1}{|s_{0}|^{2r}}\in L^1_{\rm loc}(X)\right\}  
\end{equation}
be the \emph{log-canonical threshold} of the pair $(X, X_0)$,
where $s_0$ is the section corresponding to the singular fiber $X_0$.
A very particular case of Takayama's result states as follows.

\begin{theorem}\cite{taka2}\label{assNS}
Let $u$ be a holomorphic section of $\cE_m$ defined locally near $t= 0$. Then there 
exists a constant $C(u)> 0$ such that we have
\begin{equation}\label{msri61}
|u(t)|^{2/m}_{h^{(m)}_{\cE}}\leq \frac{C(u)}{|t|^{2(1-\nu_0)}}\log^n\!\frac{1}{|t|}.
\end{equation}  
\end{theorem}
\noindent The result in \cite{taka2} is far more complete than this, e.g. the total 
space $X$ is allowed to have canonical singularities. However, we will only 
say a few words about the proof in the particular case evoked above, as follows.

One important ingredient in the arguments is that if we have a birational map
\begin{equation}
\pi:\wh X\to X
\end{equation}
and if $\wh p: \wh X\to C$ is the composed map, then the induced map
\begin{equation}
\wh p_\star(mK_{\wh X/C})\to p_\star(mK_{X/C})
\end{equation}
is an isometry. This follows directly from the definitions in the smooth case; it is still true if 
$X$ has canonical singularities, cf. \cite{taka2}. 

Thanks to this remark, given a point $x_0\in X_0$ in the singular fiber $X_0= p^{-1}(0)$ 
we can find local coordinates $(z_1,\dots, z_{n+1})$ on an open set $U\subset X$ such that 
the map $p$ corresponds to 
\begin{equation}
(z_1,\dots, z_{n+1})\to t= z_{k+1}^{b_{k+1}}\cdots z_{n+1}^{b_{n+1}}
\end{equation}   
where $k\geq 0$ is a positive integer and 
we assume that we have $1\leq b_{k+1}\leq \dots \leq b_{n+1}$. 

Let $U_t\subset \bC^n$ given by the inequalities
\begin{equation}
|z_j|^2\leq 1,\quad |t|^{2/b_l}\leq |z_l|^2\leq 1 
\end{equation}
for $j= 1,\dots k$ and for $l=k+1,\dots n$ respectively. 
We consider the projection map
\begin{equation}
\pi_{n+1}:X_t\cap U\to U_t\quad (z_1,\dots, z_n, z_{n+1})\to (z_1,\dots, z_n),
\end{equation}
which can be seen as the uniformization domain of the function
\begin{equation}\label{msri63}
z_{n+1}^{b_{n+1}}= \frac{t}{\prod_{j= k+1}^n z_j^{b_j}}. 
\end{equation} 

\noindent Let $u$ be a local section of $\cE_m$, as in Theorem \ref{assNS}. 
We interpret it as section of $p_\star(K_{X/C}+ L_m)$, where $L_m$ is the bundle
$(m-1)K_{X/C}$; on the other hand, we write
\begin{equation}\label{msri70}
u(dt^{\otimes m})= \sigma_{u, m}\left(dz^\prime\wedge p^\star(dt)\right)^{\otimes m}
\end{equation} 
so that $\sigma_{u, m}(dz^\prime)^{\otimes m}$ is a local section of 
$\displaystyle mK_{X/C}|_U$ 
corresponding to $u$, cf.\ section 3. Here  we use the notation 
$dz^\prime:= dz_1\wedge\dots\wedge dz_n$.  
%

By the definition of the metric $h^{(m)}_{X/C}$, we see that 
contribution of $X_t\cap U$ to the evaluation of \eqref{msri61} 
is bounded by 
\begin{equation}\label{msri60}
\left(\int_{X_t\cap U}|\sigma_{u,m}|^{2/m}d\lambda(z^\prime)\right)^m
\end{equation}
cf. \cite{taka2}. Next, the relation \eqref{msri70} implies that 
$$\displaystyle \sigma_{u, m}\left(z_{n+1}^{b_{n+1}-1}\prod_{j= k+1}^n z_j^{b_j}\right)^m$$
is a holomorphic function on $U$ (depending on the section $u$ and the local coordinates $(z_j)$ which we assume fixed). In particular, there exists a constant $C(u)$ 
depending on 
$u$ such that we have the 
inequality 
\begin{equation}\label{msri71}
\left\vert\frac{\sigma_{u, m}}{z_{n+1}^m}\right\vert^{2/m} \leq \frac{C(u)}{|t|^2}
\end{equation}
on $X_t\cap U$.
We therefore have to obtain an upper bound for the quantity
\begin{equation}\label{msri62}
\frac{1}{|t|^2}\int_{X_t\cap U}|z_{n+1}|^2d\lambda
\end{equation}
and by \eqref{msri63} this is equal to 
\begin{equation}\label{msri64}
\frac{1}{|t|^{2(1-1/b_{n+1})}}\int_{U_t}\frac{d\lambda}{\prod |z_j|^{2b_j/b_{n+1}}}
\end{equation}
up to a fixed constant, from which the statement \ref{assNS} follows (the 
logarithmic term appears because some of the coefficients $b_j$ 
above are equal to $b_{n+1}$).
Again, the ``real proof'' 
in \cite{taka2} is much more subtle, we have only presented here some of the main  
arguments in a simplified context.\qed
\medskip

\noindent The article \cite{taka2} contains equally an interesting corollary, arising from
algebraic geometry, cf.\ the article \cite{Fn}, Theorem 1.6,
as follows.
Assume that we have an algebraic fiber space
\begin{equation}\label{msri65}
p:X\to Y 
\end{equation}
such that the following requirements are satisfied.
\begin{enumerate}

\item[{\rm (a)}]
For any point $y\in Y$ there exists a germ of curve $(C, y)$ containing the point $y$ and such that $X_C:= p^{-1}(C)$ is smooth, and such that the 
restriction of $p$ to $X_C\sm X_y$ is a submersion. 
\smallskip

\item[{\rm (b)}] The log-canonical threshold $\nu$ corresponding to $(X_C, X_y)$ is maximal, i.e. equal to 1. 
\end{enumerate}

\noindent Then we have the following statement.

\begin{theorem}\cite{taka2}
Let $p:X\to Y$ be an algebraic fiber space satisfying the properties {\rm (a)} and 
{\rm (b)}above. We assume moreover that the 
direct image $\cE_m$ is a vector bundle, and let 
$X_m:= \bP(\cE_m^\star)$ be the corresponding projective bundle. 
Then the tautological bundle 
$\cO_{\cE_m}(1)$ has a metric with positive curvature current, and whose Lelong 
numbers at each point of $X_m$ are equal to zero.
\end{theorem}

\noindent The proof is a quick consequence of Theorem \ref{assNS} combined with 
the fact that the Lelong numbers of a closed positive current are \emph{increasing}
by restriction to a submanifold.



\end{document}